\documentclass[11pt]{amsart}

\usepackage{amsthm, amsfonts, graphicx, pinlabel}
\usepackage[all]{xy}
\usepackage{hyperref}
\usepackage[usenames,dvipsnames]{xcolor}
\usepackage{subcaption}

\include{header}

\newcommand{\leg}{\ensuremath{\Lambda}}

\def\rr{\mathbb R}

\theoremstyle{plain} 
\newtheorem{thm}{Theorem}
\newtheorem*{class-thm}{Theorem}

\newtheorem{cor}[thm]{Corollary} 
\newtheorem{prop}[thm]{Proposition}
 
\newtheorem{ques}[thm]{Question}

\newtheorem{lem}[thm]{Lemma}
 
\theoremstyle{definition} 
\newtheorem{defn}[thm]{Definition}

\theoremstyle{remark} 
\newtheorem{rem}{Remark}

\theoremstyle{Example}

\numberwithin{thm}{section}
\numberwithin{rem}{section}

\definecolor{darkmagenta}{rgb}{.5,0,.5}

\definecolor{dgr}{rgb}{0,.5, 0}
\definecolor{g2}{rgb}{.6,0, 1}

\def\Im{\operatorname{Im}}
\def\ind{\operatorname{ind}}

\begin{document}

\title[Non-Orientable Lagrangian Endocobordisms]{Non-Orientable Lagrangian Cobordisms between Legendrian Knots}

\author[O. Capovilla-Searle]{Orsola Capovilla-Searle} \address{Bryn Mawr College,
Bryn Mawr, PA 19010 and Duke University, Durham NC  27708 } \email{ocapovilla@brynmawr.edu}  
\author[L. Traynor]{Lisa Traynor} \address{Bryn Mawr College, Bryn
Mawr, PA 19010} \email{ltraynor@brynmawr.edu} 

\begin{abstract} In the symplectization of standard contact $3$-space, $\mathbb R \times \mathbb R^3$, 
it is known that an orientable Lagrangian cobordism between a Legendrian knot  and itself, also known as an orientable
Lagrangian endocobordism for the Legendrian knot,
  must have genus $0$.  We show that 
any Legendrian knot has a non-orientable Lagrangian endocobordism, and that the  crosscap genus of such a non-orientable
Lagrangian endocobordism must be a positive
multiple of $4$.  The more restrictive exact,  non-orientable Lagrangian endocobordisms
do not exist for any exactly fillable Legendrian knot but  do exist for any stabilized Legendrian knot.  Moreover, the relation defined by exact, 
non-orientable Lagrangian cobordism on the
set of stabilized Legendrian knots is symmetric
and defines an equivalence
  relation, a contrast to the non-symmetric relation defined by orientable Lagrangian cobordisms.     
 \end{abstract}

\maketitle

% ********************
\section{Introduction}
\label{sec:intro}

% **********

 Smooth cobordisms are a common object of study in topology.  
 Motivated by ideas in symplectic field theory, \cite{egh}, Lagrangian cobordisms that are cylindrical over Legendrian submanifolds outside a compact set     have been an active area of research interest.
Throughout this paper, we will study Lagrangian cobordisms  in the symplectization of the standard contact $\rr^3$, 
namely the symplectic manifold 
$(\rr \times \rr^{3}, d(e^t \alpha) )$ where $\alpha = dz - ydx$, that coincide with the cylinders $\rr \times \leg_+$ (respectively, $\rr \times \leg_-$) when the $\rr$-coordinate
is sufficiently positive (respectively, negative).  Our focus will be  on non-orientable 
Lagrangian cobordisms between Legendrian knots $\leg_+$ and $\leg_-$ and non-orientable Lagrangian endocobordisms, which are non-orientable Lagrangian
cobordisms with $\leg_+ = \leg_-$.
  
   Smooth endocobordisms in $\rr \times \rr^3$ without the Lagrangian condition are abundant:  
for any smooth knot $K \subset \rr^3$,  and an arbitrary $j \geq 0$,  there is a smooth $2$-dimensional
orientable submanifold $M$ of genus $j$  so that $M$ agrees with the cylinder $\rr \times K$ when the $\rr$ coordinate lies
outside an interval $[T_-,T_+]$; the analogous statement holds for
non-orientable $M$ and crosscap genus\footnote{the number of real projective planes in a connected sum decomposition}  when $j > 0$.    For any Legendrian knot $\leg$, it is easy to construct an orientable Lagrangian endocobordism of genus $0$, namely the
trivial Lagrangian cylinder $\rr \times \leg$.
In fact, with the added Lagrangian condition, {\it orientable} Lagrangian endocobordisms must be concordances:
 
 \begin{class-thm}[Chantraine, \cite{chantraine}]  For any Legendrian knot $\leg$, any orientable, Lagrangian endocobordism for $\leg$
    must  have genus $0$.
 \end{class-thm}

 Non-orientable Lagrangian endocobordisms also exist and have topological restrictions:
 \begin{thm}\label{thm:non-exact-endos}  For an arbitrary Legendrian knot $\leg$, there exists a 
 non-orientable Lagrangian endocobordism for $\leg$ 
 of crosscap genus $g$
 if and only if $g \in 4 \mathbb Z^+$.
 \end{thm}
 
 Theorem~\ref{thm:non-exact-endos} is proved in Theorem~\ref{thm:nole} and Theorem~\ref{thm:genus4}.
 The fact  that the crosscap genus of a non-orientable Lagrangian endocobordism must be a positive multiple of $4$ follows from  
a result of Audin about the obstruction to the Euler characteristic of closed, Lagrangian submanifolds in $\rr^4$, \cite{audin}. 
It is easy to construct {\it immersed} Lagrangian endocobordisms; the existence of the desired embedded 
endocobordisms follows from Lagrangian surgery, as developed, for example, by Polterovich in \cite{polterovich:surgery}. 

 Of special interest are Lagrangian cobordisms that satisfy an additional ``exactness" condition.  Exactness is known to be
 quite restrictive: by a foundational result of Gromov, \cite{gromov:hol}, there are no closed, exact Lagrangian submanifolds in $\rr^{2n}$ with
 its standard symplectic structure.  The non-closed trivial Lagrangian cylinder $\rr \times \leg$ is exact, and Section \ref{sec:background}
 describes  some general methods to construct exact Lagrangian cobordisms.  In contrast to Theorem~\ref{thm:non-exact-endos}, 
 there are some Legendrians that do  not admit {\it exact}, non-orientable Lagrangian endocobordisms:
 
 \begin{thm} \label{thm:fill-no-endo}  There does not exist an exact, non-orientable  Lagrangian endocobordism
 for any Legendrian knot $\leg$ that is exactly orientably or non-orientably fillable.   \end{thm}
 
 A Legendrian knot $\leg$ is exactly fillable if there exists an 
    exact   Lagrangian cobordism that is cylindrical over $\leg$ at the positive end and does not intersect $\{T_-\} \times \rr^3$,
    for $T_- \ll 0$; precise definitions can be found in Section~\ref{sec:background}. Theorem~\ref{thm:fill-no-endo} is proved in
  Section~\ref{sec:exact-obstruct}; it follows from the Seidel Isomorphism, which relates the topology of a filling
  to the linearized contact cohomology of the Legendrian at the positive end. 
  Theorem~\ref{thm:fill-no-endo} implies that 
on the set of Legendrian knots in $\rr^3$ that are exactly fillable, orientably or not, the
 relation defined by exact, non-orientable Lagrangian cobordism is {\it anti-reflexive} and {\it anti-symmetric}, see Corollary~\ref{cor:anti-symmetry}.
 Figure~\ref{fig:fillable} gives some particular examples of Legendrians that are exactly fillable and thus do not admit exact, non-orientable Lagrangian endocobordisms.  Many of these examples are maximal $tb$ Legendrian representatives of  twist and torus knots.  In fact, using the 
  classification results of Etnyre and Honda, \cite{etnyre-honda:knots}, and  Etnyre, Ng, and V\'ertesi, \cite{env}, we  show:
    
  \begin{cor}\label{cor:twist-torus}  Let $K$ be the smooth knot type of either a  twist knot or a positive torus knot  or a negative torus
  knot  of the form $T(-p, 2k)$, for $p$ odd and $p > 2k > 0$.  Then any maximal $tb$ Legendrian representative  of $K$
 does not have an exact, non-orientable Lagrangian endocobordism.
 \end{cor}

 However, stabilized Legendrian knots do admit exact, non-orientable Lagrangian endocobordisms: a Legendrian knot is said
 to be stabilized if, after Legendrian isotopy, a strand contains a zig-zag as shown in Figure~\ref{fig:stab}.
 
 \begin{thm}   \label{thm:reflexive}
 For any stabilized  Legendrian knot $\leg$ and any $k \in \mathbb Z^+$, there exists an exact, non-orientable Lagrangian
 endocobordism for $\leg$ of crosscap genus $4k$.
 \end{thm}
 
 Some Legendrian knots are neither exactly fillable nor stabilized.    Thus, a natural quetion is:
 
 \begin{ques}\label{ques:m(819)}  If a Legendrian knot is not exactly fillable and is not stabilized, does it have an exact, non-orientable Lagrangian
 endocobordism?  In particular, does the Legendrian representative of $m(8_{19})=T(-4,3)$ with maximal $tb$ shown in Figure~\ref{fig:m819} have an exact,
 non-orientable Lagrangian endocobordism?  
  \end{ques} 
 \noindent
The max $tb$ version of $m(8_{19})$ is not exactly fillable since the upper bound on  
the $tb$ invariant for all Legendrian representatives of $m(8_{19})$ 
given by the Kauffman polynomial is not sharp; Section~\ref{sec:questions} for more details and related questions.    
 
\begin{figure}
 \centerline{\includegraphics[height=.6in]{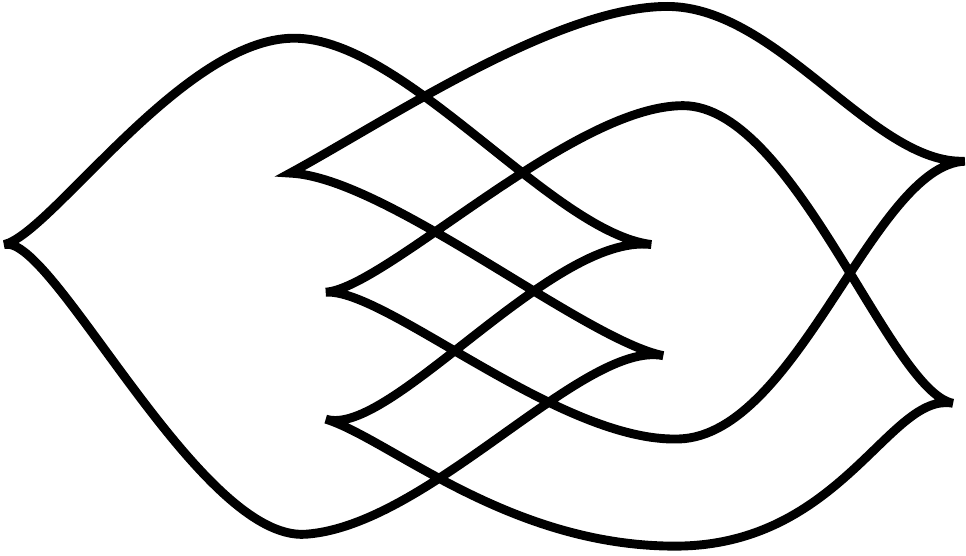}}
  \caption{Does the max $tb$ Legendrian representative of $m(8_{19})$
   have an exact, non-orientable
  Lagrangian endocobordism?}
   \label{fig:m819}
\end{figure}

Given the existence of exact, non-orientable Lagrangian endocobordisms for a stabilized Legendrian, it is natural to ask:  
   What Legendrian knots can appear as a ``slice" of such an endocobordism?  
   The parallel question for orientable Lagrangian endocobordisms has been studied in
\cite{chantraine:non-symm, bs:monopole, cns:obstr-concordance}. The non-orientable version of this question is closely tied to the question
   of whether non-orientable Lagrangian cobordisms define an equivalence relation on the set of Legendrian knots.
 By a result of Chantraine, \cite{chantraine},  it is known that the relation defined on the set of Legendrian knots by {\it orientable} Lagrangian cobordism is {\it not} an equivalence relation since
 symmetry fails.  In fact, the relation defined on the set of {\it stabilized} Legendrian knots by exact, non-orientable Lagrangian cobordism is symmetric:  see Theorem~\ref{thm:different-knots-path}. It is then easy to deduce:

 \begin{thm}\label{thm:equiv}  On the set of  stabilized Legendrian knots, the relation defined by exact, non-orientable Lagrangian cobordism is an equivalence relation.  Moreover, all stabilized Legendrian knots are equivalent with respect to this relation.
 \end{thm}

 \subsection*{Acknowledgements}   We thank Baptiste Chantraine, Richard Hind, and Josh Sabloff for stimulating discussions. We also
 thank  Georgios  Dimitroglou Rizell and Tobias Ekholm for helpful comments.   
 Both authors thank the Mellon-Mays Foundation for supporting the first author with a Mellon-Mays Undergraduate Fellowship; this paper 
 grew out of her  thesis project, \cite{CS:loops}.

\section{Background}\label{sec:background}

In this section, we give some basic background on Legendrian and Lagrangian submanifolds.

\subsection{Contact Manifolds and Legendrian Submanifolds}

Below is some basic background on contact manifolds and Legendrian knots.  More information can be found, for example, in
\cite{etnyre:intro} and \cite{etnyre:knot-intro}.
  
 A {\bf contact manifold} $(Y, \xi)$ is an odd-dimensional manifold together with a contact structure, which consists
 of a field of maximally non-integrable tangent hyperplanes.
The {\bf standard contact
structure} on $\mathbb R^3$ is the field  
$\xi_p = \ker \alpha_0(p)$, for  $\alpha_0(x,y,z) = dz - ydx$.  A {\bf Legendrian link} is a submanifold, $\leg$, of $\mathbb R^3$
diffeomorphic to a disjoint union of circles so that for all $p \in \leg$,
$T_p\leg \subset \xi_p$; if, in addition, $\leg$ is connected, $\leg$ is a {\bf Legendrian knot}.
It is common to examine Legendrian links from their
$xz$-projections, known as their {\bf front projections}.  A Legendrian link will generically have
an immersed front projection with semi-cubical cusps and no vertical tangents; any such projection can
be uniquely lifted to a Legendrian link using $y = dz/dx$.

Two Legendrian links $\Lambda_0$ and $ \Lambda_1$ are {\bf equivalent Legendrian links}
if there exists a $1$-parameter family of Legendrian links $\Lambda_t$ joining $\Lambda_0$ and $\Lambda_1$. 
In fact, Legendrian links $\Lambda_0, \Lambda_1$ are equivalent if and only if their front projections
are equivalent by planar isotopies that do not introduce vertical tangents  and 
the {\bf Legendrian Reidemeister moves} as  shown in Figure~\ref{fig:L-R-moves}.
\begin{figure}
  \centerline{\includegraphics[height=1.2in]{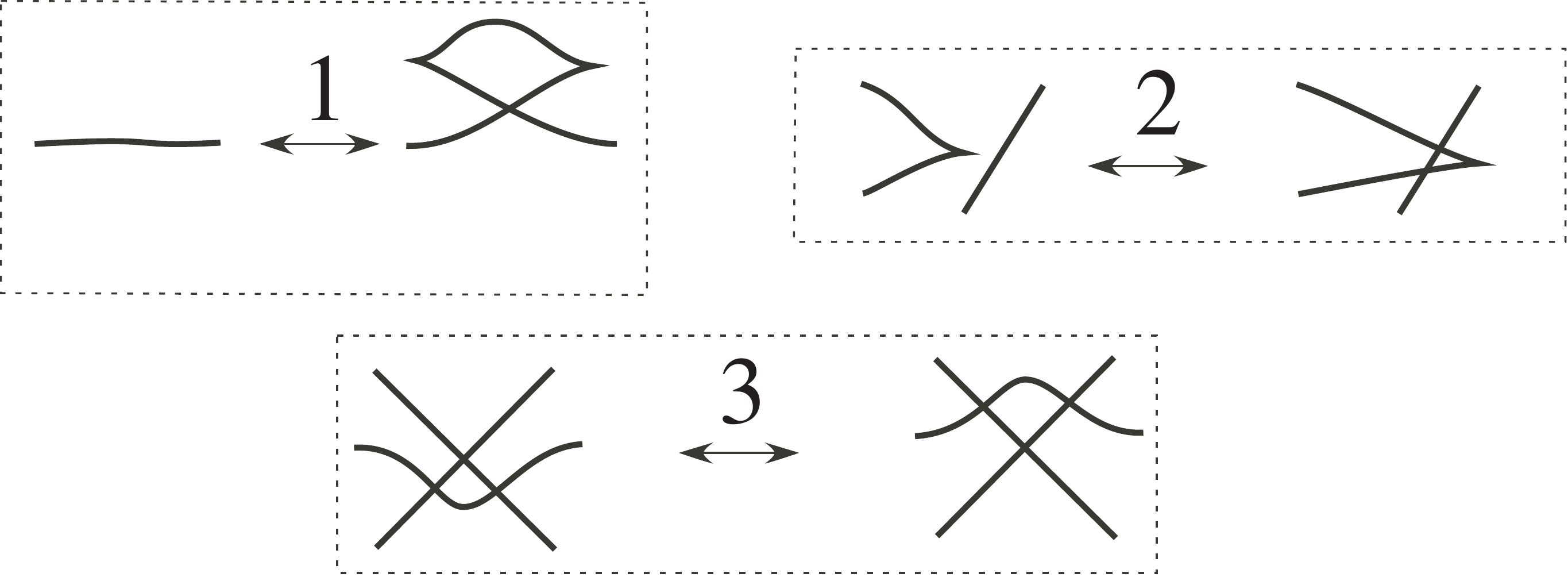}}
  \caption{The three Legendrian Reidemeister moves.
  }
  \label{fig:L-R-moves}
\end{figure}

Every Legendrian knot has a Legendrian representative.  In fact, every Legendrian knot
 has an infinite number of different Legendrian representatives.  For example, Figure~\ref{fig:3-unknots}
shows three different oriented Legendrians that are all topologically the unknot.  These unknots can be distinguished
by classical Legendrian invariant numbers, the Thurston-Bennequin, $tb$, and rotation, $r$.
These invariants can easily be computed from a front projection; see, for example, \cite{bty}.
    \begin{figure}
  \centerline{\includegraphics[height=.8in]{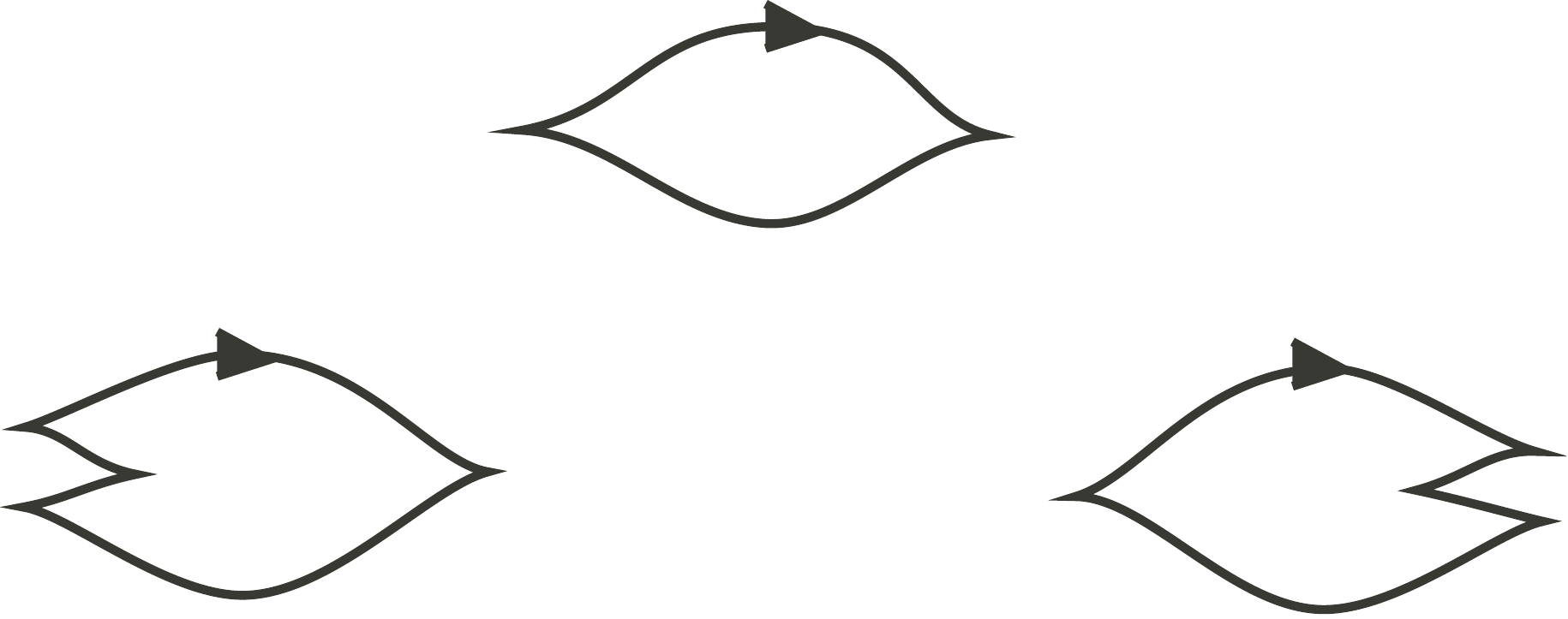}}
  \caption{Three different Legendrian unknots; the one with maximal $tb$ invariant
  of $-1$ and two others obtained by $\pm$-stabilizations. } 
  \label{fig:3-unknots}
\end{figure}

 The two unknots in  the second line of Figure~\ref{fig:3-unknots} are obtained from the one at the top
 by stabilization.   In general, from an oriented Legendrian $\leg$, one can obtain oriented 
 Legendrians $S_\pm(\leg)$:   the {\bf positive (negative) stabilization, $S_+$ ($S_-$),}  is obtained by replacing a portion of
 a strand with a strand that contains a down (up) {\bf zig-zag}, as shown in Figure~\ref{fig:stab}.  
  This stabilization procedure will not  change the underlying smooth knot type but will
 decrease the Thurston-Bennequin number by $1$; adding an up (down) zig-zag will decrease (increase)
 the rotation number by $1$.   It is possible to move a zig-zag to any strand of a Legendrian knot, \cite{f-t}.
  For any smooth knot type, all Legendrian representatives
 can be represented by a mountain range that records the possible $tb$ and $r$ values; many examples of known and conjectured mountain
 ranges can be found in the Legendrian knot atlas of Chongchitmate and Ng, \cite{ng:atlas}.

  \begin{figure}
 \centerline{\includegraphics[height= .8in]{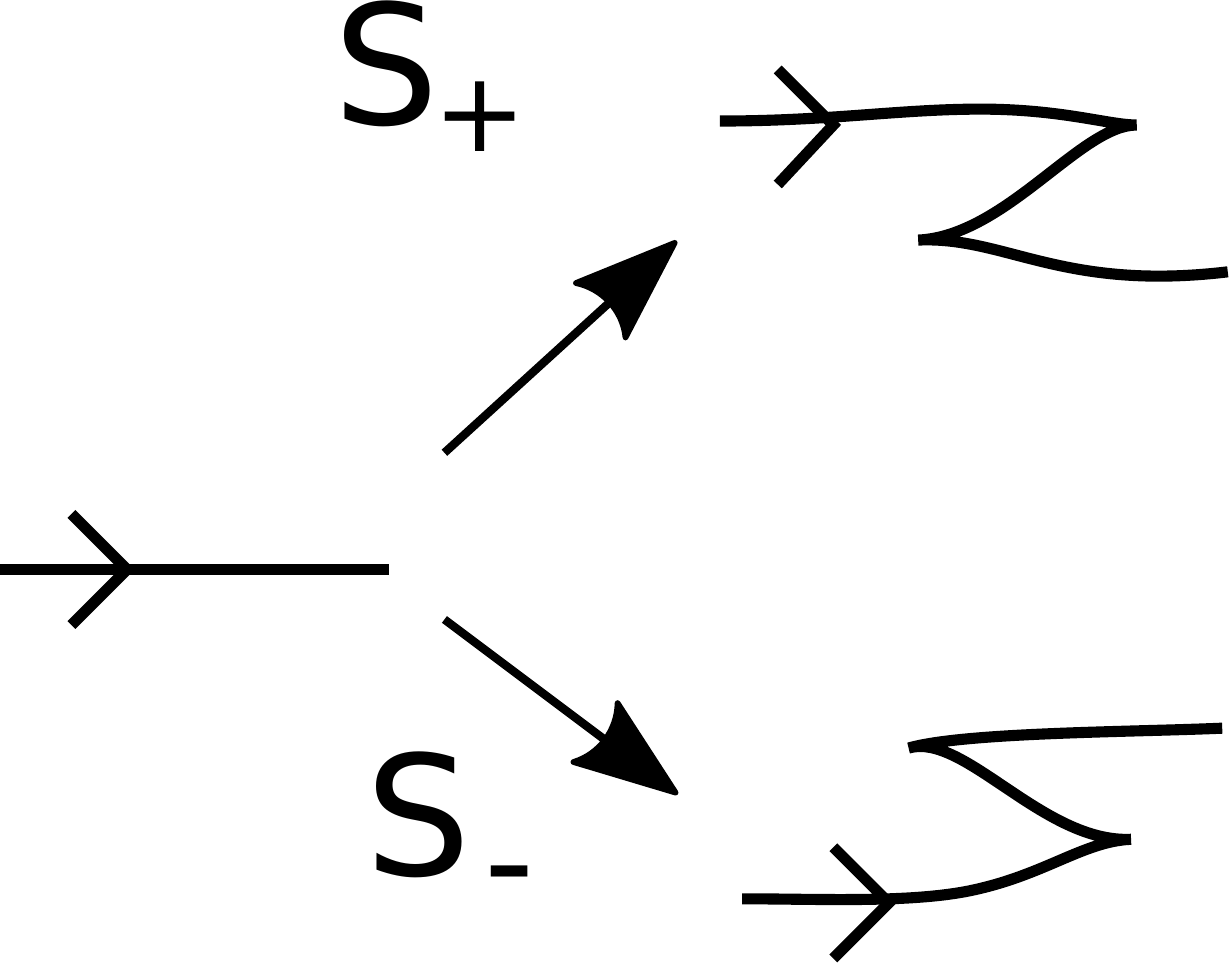}}
  \caption{The positive (negative) stabilization of an oriented knot is obtained by 
  introducing a down (an up) zig-zag.}
  \label{fig:stab}
\end{figure}

\subsection{Symplectic Manifolds, Lagrangian Submanifolds, and Lagrangian Cobordisms}\label{ssec:lagrangians}
 We will now discuss some basic concepts in symplectic geometry.  Additional background can be found, for example,
 in \cite{mcduff-salamon}.
 
 A {\bf symplectic manifold} $(M, \omega)$ is an even-dimensional manifold together with a $2$-form $\omega$
 that is closed and non-degenerate; when $\omega$ is an exact $2$-form, $(M, \omega=d\beta)$ is said to
 be an {\bf exact symplectic manifold}.  A basic example of an exact symplectic manifold is
 $(\rr^4, \omega_0 = dx_1 \wedge dy_1 + dx_2 \wedge dy_2)$.  
The cobordisms constructed in this paper live inside the symplectic manifold that is
constructed as the symplectization of $(\rr^3, \xi_0 = \ker \alpha_0)$, namely,
$\rr \times \rr^3$ with symplectic form given by $\omega = d(e^t \alpha_0)$.
In fact, the symplectization $(\rr \times \rr^3, \omega)$ is exactly symplectically equivalent to the standard
$(\rr^4, \omega_0)$, see for example \cite{bst:construct}.

A {\bf Lagrangian submanifold} $L$ of a $4$-dimensional symplectic manifold $(M, \omega)$
is a $2$-dimensional submanifold so that $\omega|_L = 0$.  When $M$ is an exact symplectic manifold,
$\omega = d\beta$, $\beta|_L$ is necessarily a closed $1$-form; when, in addition, $\beta|_L$ is
an exact $1$-form, $\beta|_ L = df$, then $L$ is said to be an {\bf exact Lagrangian submanifold}.

\begin{rem}  \label{rem:torus} There is a (non-exact) Lagrangian torus in the standard symplectic $\rr^4$: this can be seen as the product of two embedded circles
in each of the $(x_1,y_1)$ and $(x_2, y_2)$ planes.  By  classical algebraic topology, it follows 
that the torus is the only compact, orientable surface that admits a Lagrangian embedding into $\rr^4$,
\cite{audin-lalonde-polterovich}. 
\end{rem}
 
We will focus on non-compact Lagrangians that are  cylindrical over Legendrians.

\begin{defn}  Let $\leg_-, \leg_+$ be Legendrian links in   $\rr^3$.
\begin{enumerate}
\item A Lagrangian submanifold without boundary  $L \subset \rr \times \rr^3$  is  a {\bf Lagrangian cobordism from $\leg_+$ to $\leg_-$} if it is of the form
$$L = \left( (-\infty, T_-] \times \leg_- \right) \cup \overline L \cup \left( [T_+, +\infty) \times \leg_+ \right),$$
for some $T_- < T_+$, where $\overline L \subset [T_-, T_+] \times \rr^3$ is compact with boundary
$\partial \overline L =  \left( \{ T_- \} \times \leg_-  \right) \cup  \left( \{ T_+ \} \times \leg_+\right)$.
\item  A {Lagrangian cobordism from $\leg_+$ to $\leg_-$} is {\bf orientable (resp., non-orientable)}
if $L$ is orientable (resp., non-orientable).
\item  A {Lagrangian cobordism from $\leg_+$ to $\leg_-$} is {\bf exact} if $L$ is exact,  namely $e^t \alpha_0|_L = df|_L$,
and the primitive, $f$, is constant on the cylindrical ends: 
there exists constants $C_\pm$ so that
$$f|_{L \cap ((-\infty, T_-) \times \rr^3)}= C_-, \qquad f|_{L \cap ((T_+, +\infty) \times \rr^3)}= C_+.$$
 \end{enumerate}
 A Legendrian knot $\leg$ is {\bf (exactly) fillable} if there exists an (exact) Lagrangian cobordism
from $\leg_+ = \leg$ to $\leg_- = \emptyset$.
  \end{defn}

An important property of Lagrangian cobordisms is that they can be stacked/composed:

\begin{lem}[Stacking Cobordisms, \cite{ehk:leg-knot-lagr-cob}] \label{lem:gluing}If $L_{12}$ is an exact Lagrangian cobordism from $\leg_+ = \leg_1$ to $\leg_- = \leg_2$,
and $L_{23}$ is an exact Lagrangian cobordism from $\leg_+ = \leg_2$ to $\leg_- = \leg_3$, then there exists an
exact Lagrangian cobordism $L_{13}$ from $\leg_+ = \leg_1$ to $\leg_- = \leg_3$.
\end{lem}

Constructions of exact Lagrangian cobordisms are an active area of research.  In this paper, we will use the
fact that there exist 
exact Lagrangian cobordisms between Legendrians related by isotopy and surgery.  The existence of
exact Lagrangian cobordisms from isotopy is well-known, 
see, for example,  \cite{eg:finite-dim}, \cite{chantraine},  \cite{ehk:leg-knot-lagr-cob}, and \cite{bst:construct}.   
 
\begin{lem}[Exact Cobordisms from Isotopy] \label{lem:isotopy} Suppose that $\leg$ and $\leg'$ are isotopic Legendrian knots.  Then there exists an exact, orientable Lagrangian cobordism
from $\leg_+ = \leg$ to $\leg_- = \leg'$.
\end{lem}

\begin{rem} \label{rem:isotopy} In general, the trace of a Legendrian isotopy is not  
a Lagrangian cobordism.  However it is possible to add a ``correction term" so that it will be Lagrangian.  More
precisely, 
let $\lambda_t(u) = (x(t, u), y(t, u), z(t, u))$, $t \in \mathbb R$, be a Legendrian isotopy so that  
$\frac{\partial \lambda}{\partial t} (t,u)$ has
compact support  with  $\Im \lambda_t(u) = \leg_-$ for $t \leq -T$ and  $\Im \lambda_t(u) = \leg_+$ for $t \geq T$, and  let 
$$\eta(t, u) = \alpha_0\left(\frac{\partial \lambda}{\partial t} (t, u)\right).$$  Then
$\Gamma(t, u) = (t, x(t, u), y(t, u), z(t, u) + \eta(t, u))$ is an exact Lagrangian immersion.  If $\eta(t, u)$ is
sufficiently small, which can be guaranteed by making $T$ sufficiently large, then $\Gamma(t,u)$ is an exact Lagrangian embedding.
\end{rem}

In addition, Legendrians $\leg$ and $\leg'$ that differ by ``surgery" can be connected by an exact Lagrangian cobordism. 
 The $0$-surgery operation can be viewed as a ``tangle surgery": the replacement
 of a Legendrian $0$-tangle, consisting of two strands with no crossings and no cusps,
 with a Legendrian $\infty$-tangle, consisting of two strands that each have 1 cusp and no crossings; see Figure~\ref{fig:o-no-surgery}.
 When the strands of the $0$-tangle are oppositely oriented, this is an {\bf orientable surgery}; 
 otherwise this is a {\bf non-orientable surgery}.  In addition, by an index $1$ surgery, it is known that 
 the maximal $tb$ Legendrian representative of the unknot, shown at the top of Figure~\ref{fig:3-unknots}, can be filled.

 \begin{figure}
   \centerline{\includegraphics[height=1in]{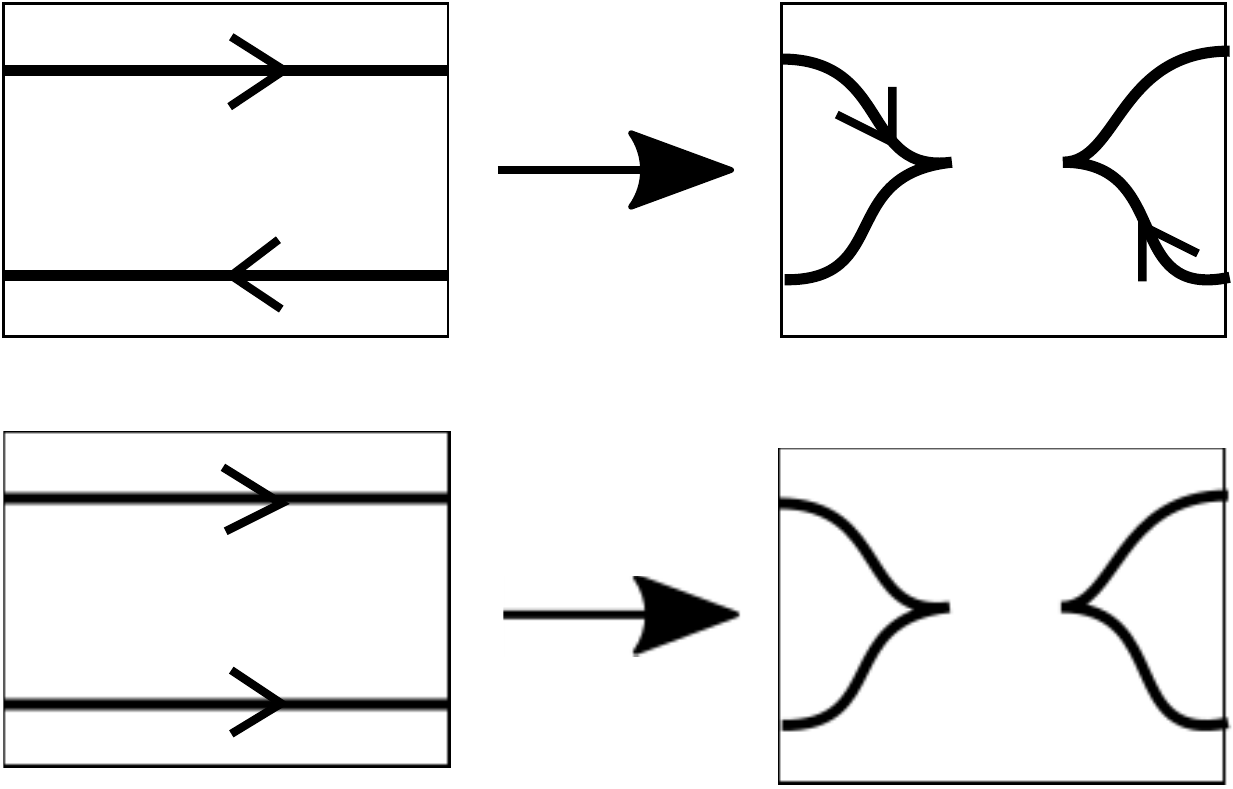}}
  \caption{Orientable and Non-Orientable Legendrian surgeries. }
  \label{fig:o-no-surgery}
\end{figure}

 \begin{lem}[Exact Cobordisms from Surgery, \cite{ehk:leg-knot-lagr-cob, rizell:surgery, bst:construct}]\label{lem:surgery}
  \begin{enumerate}
  \item Suppose that $\leg_+$ and $\leg_-$ are Legendrian knots where $\leg_-$ is obtained from $\leg_+$ by
  orientable (non-orientable) surgery, as shown in Figure~\ref{fig:o-no-surgery}.  Then there exists an exact, orientable (non-orientable)
  Lagrangian cobordism from $\leg_+$ to $\leg_-$.
  \item  Suppose $\leg_+$ is the Legendrian unknot with $tb$ equal to the maximum value of $-1$.  Then there exists
  an exact, orientable Lagrangian filling of $\leg_+$.
  \end{enumerate}
 \end{lem}

\begin{rem} \label{rem:isotopy+surgery} By Lemmas
~\ref{lem:isotopy} and \ref{lem:surgery}, to show there exists an exact Lagrangian cobordism from $\leg_+$ to $\leg_-$, 
 it suffices to show that there is a string of Legendrian
links $(\leg_+ = \leg_0 , \leg_1 , \dots, \leg_n = \leg_-)$, 
where each $\leg_{i+1}$ is obtained from $\leg_i$ by a single surgery, as shown in Figure~\ref{fig:o-no-surgery}, 
and Legendrian isotopy.  In the case where each surgery is orientable, the exact Lagrangian cobordism
 will be orientable; in this case, the length $n$ of this string must be even and will agree
with twice the genus of the Lagrangian cobordism; for more details, see \cite{bty}. If there is at least one non-orientable surgery, the 
exact Lagrangian cobordism will be non-orientable and the length of the string agrees with the crosscap genus of the Lagrangian cobordism. 
To construct an exact Lagrangian filling of $\leg_+$, it suffices to construct such a string to $\leg_- = U$,
where $U$ is a trivial link of maximal $tb$ Legendrian unknots. 
\end{rem}

 \section{Constructions of Non-orientable Lagrangian Endocobordisms}

 In this section, we show that {\it any} Legendrian knot has a non-orientable Lagrangian endocobordism with
 crosscap genus an arbitrary multiple of $4$.  We then show that it is not possible to get any other crosscap genera.

The strategy to show existence is to first construct an immersed orientable Lagrangian cobordism, and then
apply ``Lagrangian surgery" to modify it so that it is embedded.  The following
 description of Lagrangian surgery follows Polterovich's construction, \cite{polterovich:surgery};
see also work of Lalonde and Sikorav, \cite{lalonde-sikorav}. 

To state Lagrangian surgery precisely,   we first
need to explain the ``index" of a double point.   Suppose that $x$ is a point of self-intersection of a generic, immersed, oriented $2$-dimensional submanifold
$L$ of $\rr^4$.  Then  $ind(x) \in \{ \pm 1 \}$ will denote the {\bf index of self-intersection of $L$} at $x$:  let $(v_1, v_2)$ and
$(w_1, w_2)$ be positively oriented bases of the transverse tangent spaces at $x$, then 
$$\ind(x) = +1  \iff (v_1, v_2, w_1, w_2) \text{ is a positively  oriented basis of } \rr^4,$$
and otherwise $\ind(x) = -1$.

By constructing a Lagrangian handle in a Darboux chart, it is possible to remove double points of a Lagrangian:

\begin{lem}[Lagrangian Surgery,  \cite{polterovich:surgery}] \label{lem:lag-surgery}
Let $\Sigma$ be a 2-dimensional manifold.  Suppose $\phi: \Sigma \to \rr^4$ is a
Lagrangian immersion, and $U \subset \rr^4$ contains a single transversal double point  $x$ of $\phi$.
Then there exists a 2-dimensional manifold $\Sigma'$, which is obtained by a  
Morse surgery
on $\Sigma$, and a Lagrangian immersion $\phi': \Sigma' \to \rr^4$
so that 
\begin{enumerate}
\item  $\Im \phi = \Im \phi'$ on $\rr^4 - U$;
\item  $\phi'$ has no double points in $U$.
\end{enumerate}
Furthermore, let $\phi^{-1}(\{ x \}) = \{ p_1, p_2 \} \subset \Sigma$.  Then
\begin{enumerate}
\item  if $p_1, p_2$ are in disjoint components of $\Sigma$, then $\Sigma'$ is obtained
from $\Sigma$ by a connect sum operation;
\item if $p_1, p_2$ are in the same component of $\Sigma$ then:
\begin{enumerate}
\item if $\Sigma$ is not oriented, $\Sigma' = \Sigma \# K$,
\item if  $\Sigma$ is oriented, then $\Sigma' = \Sigma \# T$, when
$\ind(x) = +1$, and $\Sigma' = \Sigma \# K$, when $\ind(x) = -1$,
\end{enumerate}
where $K$ denotes the Klein bottle, and $T$ denotes the torus.
\end{enumerate}
\end{lem}

 We now have the necessary background to show the existence of a non-orientable Lagrangian
 endocobordism for any Legendrian knot:
 
 \begin{thm}\label{thm:nole}  For any Legendrian knot $\leg$ and any $k \in \mathbb Z^+$, there exists a 
 non-orientable Lagrangian endocobordism for $\leg$  
 of crosscap genus $4k$.
 \end{thm}

\begin{proof}  For an arbitrary Legendrian knot $\leg$,  begin with cylindrical Lagrangian cobordism, 
$L =  \rr \times \leg  
\subset \rr \times \rr^3$, which is a space that is symplectically equivalent to the standard $\mathbb R^4$.  
As explained in   Remark~\ref{rem:torus},
there exists  an  embedded Lagrangian torus, $T$, so that $T \cap L = \emptyset$.  After a suitable shift and perturbation,
we can assume that $L$ and $T$ intersect at exactly two points, $x_1$ and $x_2$ where $\ind(x_1) = +1$ and
$\ind(x_2) = -1$.  By Lemma~\ref{lem:surgery}, Lagrangian surgery at $x_1$ results in the connected, oriented, immersed Lagrangian
 diffeomorphic to $(\rr \times S^1) \# T$ with a double point at $x_2$ of index $-1$; a second Lagrangian surgery at $x_2$ results in a
  embedded, non-orientable Lagrangian cobordism
 diffeomorphic to $\rr \times S^1 \times T \times K$, and thus of crosscap genus $4$. 
Stacking these endocobordisms, using Lemma~\ref{lem:gluing},    
produces an embedded, non-orientable Lagrangian cobordism of crosscap genus $4k$, for any $k \in \mathbb Z^+$.
\end{proof}

 In fact, the possible  crosscap genera that appeared in Theorem~\ref{thm:nole} are all that can exist:
 
 \begin{thm} \label{thm:genus4} Any non-orientable Lagrangian endocobordism  in  $\rr \times \rr^3$
  must have
 crosscap genus $4k$, for some $k \in \mathbb Z^+$. 
 \end{thm}

This crosscap genus restriction is closely tied to Euler characteristic obstructions for {\it compact}, non-orientable submanifolds
that admit  Lagrangian embeddings in $(\rr^4, \omega_0)$, or equivalently in $(\rr \times \rr^3, d(e^t\alpha))$:

\begin{lem}[Audin, \cite{audin}] \label{lem:no-euler} Any compact, non-orientable Lagrangian submanifold of $\rr \times \rr^3$
has an Euler characteristic divisible by $4$.  
\end{lem}

This result can be seen as an extension of a formula of Whitney that relates the number of double points of a smooth
immersion to the Euler characteristic of the normal bundle of the immersion and thus of the 
tangent bundle of a Lagrangian immersion; see \cite{audin, audin-lalonde-polterovich}.
 
\begin{rem}  Lemma~\ref{lem:no-euler} implies that any compact, non-orientable,  Lagrangian submanifold $L$ in $\rr \times \rr^3$ 
has crosscap genus $2+4j$, for some $j \geq 0$.
There are explicit constructions of compact, non-orientable Lagrangian submanifolds of
crosscap genus $2 + 4j$, for all $j > 0$, \cite{givental:construct, audin:givental}.  It has been shown that there is no embedded, Lagrangian Klein bottle ($j = 0$),
\cite{nemirovski:klein, shevchishin:klein}.
\end{rem}

To utilize the crosscap genus restrictions for compact Lagrangians, we will employ the 
 following lemma, which shows that for any Lagrangian endocobordism,
 it is possible to construct a compact, non-orientable Lagrangian submanifold into which
we can glue the compact portion of a Lagrangian  endocobordism.

\begin{lem} \label{lem:cpct-closure} For any Legendrian knot $\leg \subset \rr^3$, any open set $D \subset \rr^3$ containing $\leg$,
and any $T \in \rr^+$, there exists
a compact, non-orientable Lagrangian submanifold $L$ in $\rr \times \rr^3$ so that 
$$L \cap ([-T , T] \times D) = [-T , T] \times \leg.$$
\end{lem}

\begin{proof}  The strategy will be to construct a Lagrangian immersion of the torus, thought of as two finite cylinders with
top and bottom circles identified, and then apply Lagrangian surgery to remove the immersion points.  
As a first step, we construct (non-disjoint) Lagrangian embeddings of two cylinders via Legendrian isotopies, Lemma~\ref{lem:isotopy}.
Namely, start with two disjoint copies of $\leg$: $\leg$ in $D$ and a translated version $\leg' \in \rr^3 - D$. 
Now, for $t \in [0, t_2]$, 
consider Legendrian isotopies $\leg_t$ of $\leg$ and $\leg_t'$ of $\leg'$  that satisfy the following conditions:
$\leg_t = \leg$, for all $t \in [0, t_2]$; $\leg_t' = \leg'$, for $t \in [0, t_1]$,  and then for 
$t \in [t_1, t_2]$, $\leg_t'$ is a Legendrian isotopy of  $\leg'$ so that $\leg_{t_2}' = \leg = \leg_{t_2}$.   
By repeating an analogous procedure for $t \in [-t_2, 0]$, we can obtain a smooth, immersion of the torus into $[-t_2, t_2] \times \rr^3$.  
The arguments used to prove Lemma~\ref{lem:isotopy} (see Remark~\ref{rem:isotopy}) show that for sufficiently large $t_2$, the image of the trace
of these isotopies can be perturbed to  two non-disjoint
embedded Lagrangian cylinders that do not have any intersection points in $[-t_1, t_1] \times \rr^3$.
 Then by applying Lagrangian surgery, Lemma~\ref{lem:surgery}, at each double point we get a 
compact, non-orientable Lagrangian submanifold $L$ in $\rr \times \rr^3$ with the desired properties.
\end{proof}

We are now ready to prove the crosscap genus restriction for arbitrary non-orientable, Lagrangian endocobordisms:

\begin{proof}[Proof of Theorem~\ref{thm:genus4}] Let $C$ be a non-orientable Langrangian endocobordism.  Suppose $C \subset \rr \times D$
and $C$ agrees with standard cylinder outside $[-T, T] \times \rr^3$.  By Lemma~\ref{lem:cpct-closure},  there is a compact, non-orientable 
Lagrangian submanifold $L$ in $\rr \times \rr^3$ so that 
$$L \cap ([-T, T] \times D) = [-T, T] \times \leg.$$  Let $L'$ be the Lagrangian submanifold obtained by removing the
standard cylindrical portion of $L$ in $[-T , T] \times D$ and replacing it with $C \cap ([-T, T] \times \rr^3)$.  Then $L'$ will be a  compact,
non-orientable  Lagrangian submanifold whose crosscap genus, $k(L')$, differs from the crosscap genus of $L$, $k(L)$,
 by the crosscap genus of  $C$, $k(C)$:
 $k(L') = k(L) + k(C)$.
  By Lemma~\ref{lem:no-euler},  there exist $j, j' \in \mathbb Z^+$ so that 
$k(L) = 2 + 4j$ and $k(L') = 2 + 4j'$.  Thus we find that  the crosscap genus of $C$,
$k(C)$, must
be divisible by $4$. 
 \end{proof}
 
 \begin{rem} For exact Lagrangian cobordisms that are constructed from isotopy and surgery, Lemmas ~\ref{lem:isotopy} and \ref{lem:surgery},
 it is possible to show that the crosscap genus must be a multiple of $4$ by an alternate argument
 that relies on a careful analysis of the possible changes to $tb(\leg)$ under surgery;
 \cite{CS:loops}.
 \end{rem}

 \section{Obstructions to Exact Non-Orientable Lagrangian Endocobordisms} \label{sec:exact-obstruct}
 
 We will now begin to focus on {\it exact}, non-orientable Lagrangian cobordisms.  In this section, we will prove
 Theorem~\ref{thm:fill-no-endo}, which states that any
 Legendrian knot that is exactly fillable does not have an exact non-orientable Lagrangian endocobordism.  The
 proof of this theorem will involve applying the Seidel Isomorphism, which relates the the topology of a filling to
 the linearized Legendrian contact cohomology of the Legendrian at the positive end.  We will then apply 
 Theorem~\ref{thm:fill-no-endo} and give examples of maximal $tb$ Legendrian knots that do not have
 exact, non-orientable Lagrangian endocobordisms.
   
 We begin with a brief description of Legendrian contact homology; additonal background
 information can be found, for example, in \cite{etnyre:knot-intro}.
  Legendrian contact homology is a Floer-type invariant of a Legendrian submanifold that
 lies within Eliashberg, Givental, and Hofer's Symplectic Field Theory framework; \cite{yasha:icm, egh, chv}.
It is possible to associate to a Legendrian submanifold $\leg \subset \rr^3$
 the stable, tame isomorphism class of an associative differential graded algebra (DGA),
 $(\mathcal A(\leg), \partial)$.  The algebra is freely generated by the Reeb chords of $\leg$,
 and is graded using a Maslov index.  The differential comes from counting pseudo-holomorphic
 curves in the symplectization of $\rr^3$; for our interests, we will always use $\mathbb Z/2$ coefficients.
  {\bf Legendrian contact homology}, namely the homology of $(\mathcal A(\leg), \partial)$, is
 a Legendrian invariant of $\leg$.
 
 In general, it is difficult to extract information directly from the Legendrian contact homology.  An important computational technique arises from the existence of augmentations of the DGA.
 An   
 {\bf augmentation} $\varepsilon$ of $\mathcal A(\leg)$ is a differential  algebra homomorphism
 $\varepsilon: (\mathcal A(\leg), \partial) \to (\mathbb Z_2, 0)$;
 a {\bf graded augmentation} is an augmentation so that $\varepsilon$ is
 supported on elements of degree $0$.  Observe that, for any Legendrian $\leg$, there are only
 a finite number of augmentations.
  Given a graded augmentation $\varepsilon$, one can linearize 
 $(\mathcal A(\leg), \partial)$ to a finite dimensional differential graded complex $(A(\leg), \partial^\varepsilon)$
 and obtain {\bf linearized contact homology}, denoted $LCH_*(\leg, \varepsilon; \mathbb Z/2)$, and
 its dual {\bf linearized contact cohomology},  $LCH^*(\leg, \varepsilon; \mathbb Z/2)$.  The set of all linearized (co)homology
 groups with respect to all possible graded augmentations is an invariant of $\leg$.  If the augmenation is ungraded,
 one can still examine the rank of the non-graded linearized (co)homology, $\dim LCH(\leg, \varepsilon; \mathbb Z/2)$, 
 and obtain as an invariant of $\leg$ the
 set  of  ranks of this total linearized (co)homology for all possible augmentations.  Examining ungraded linearized (co)homology
 is not an effective invariant: of the many examples of Legendrians knots in the Legendrian knot atlas of Chongchitmate and Ng, \cite{ng:atlas}, that have the
 same classical invariants
   yet  
 can be distinguished through graded Linearized homology, none of these can be distinguished  by examining
 ungraded homology.  However, ungraded (co)homology will be useful in arguments below.

  Ekholm, \cite{ekholm:rsft}, has shown that an exact Lagrangian filling, $F$, of a Legendrian submanifold $\leg \subset \rr^3$ induces an
  augmentation $\varepsilon_F$ of $(\mathcal A(\leg), \partial)$.  When this filling  has Maslov class $0$, the augmentation
  will be graded.  
   
The following result of Seidel will play a central role in showing obstructions to exact, non-orientable Lagrangian endocobordisms.  
  A proof of this result was sketched by Ekholm in  
 \cite{ekholm:lagr-cob} and
given in detail in Dimitroglou-Rizell, \cite{rizell:lifting}; a parallel result using generating family homology is given in \cite{josh-lisa:obstr}.

  \begin{thm}[Seidel Isomorphism, \cite{ekholm:lagr-cob}, \cite{rizell:lifting}, \cite{ehk:leg-knot-lagr-cob}]
   \label{thm:seidel} Let $\leg \subset \rr^3$ Legendrian submanifold
  with an exact Lagrangian filling $F$; let $\varepsilon_F$ denote the augmentation induced by the filling.  Then
 $$\dim H(F; \mathbb Z/2) = \dim LCH(\leg, \varepsilon_F; \mathbb Z/2).$$
  If the  filling $F$ of the $n$-dimensional Legendrian  has Maslov class $0$, then a graded version of the above equality holds:
 $$\dim H_{n-*}(F; \mathbb Z/2) = \dim LCH^*(\leg, \varepsilon_F; \mathbb Z/2).$$
 \end{thm}
 
 The ungraded version of the Seidel Isomorphism will be used to prove that any Legendrian $\leg$ that
 is exactly fillable does not have an exact, non-orientable Lagrangian endocobordism:

 \begin{proof}[Proof of Theorem~\ref{thm:fill-no-endo}]  For a contradiction, suppose that there is a Legendrian knot $\leg$ that
 has an exact Lagrangian filling and an exact non-orientable Lagrangian endocobordism.  Then
 by stacking the endocobordisms, Lemma~\ref{lem:gluing},  it follows that $\leg$ has an infinite number of topologically distinct
 exact, non-orientable Lagrangian fillings.  Each of these exact Lagrangian fillings induces an augmentation.  Since
 there are only a finite number of possible augmentations, there must exist two topologically distinct
 fillings that induce the same augmentation.  However, this gives a contradiction to the Seidel Isomorphism, Theorem~\ref{thm:seidel}.
    \end{proof}

 Theorem~\ref{thm:fill-no-endo} implies that 
on the set of Legendrian knots in $\rr^3$ that are exactly fillable, orientably or not, the
 relation  defined by exact, non-orientable Lagrangian cobordism is anti-reflexive.  Thus, by stacking, Lemma~\ref{lem:gluing},
 we immediately also see:

 \begin{cor} \label{cor:anti-symmetry} On the set of Legendrian knots in $\rr^3$ that are exactly fillable, orientably or not, the
 relation $\sim$ defined by exact, non-orientable Lagrangian cobordism is anti-symmetric:
 $\leg_1 \sim \leg_2 \implies \leg_2 \not\sim \leg_1.$
 \end{cor}

 We now apply Theorem~\ref{thm:fill-no-endo} to give examples of Legendrians that do not have
 exact, non-orientable Lagrangian endocobordisms.
  Hayden and Sabloff, \cite{positivity}, showed that every positive knot type has a Legendrian representative
 that has an exact, orientable Lagrangian filling.  In addition,  Lipman, Reinoso, and Sabloff have shown that
 every $2$-bridge knot and every $+$-adequate knot has a Legendrian representative with an exact filling, \cite{lrs:filling}.
 Combining this with Theorem~\ref{thm:fill-no-endo}, immediately gives:
    
 \begin{cor}[\cite{positivity}, \cite{lrs:filling}]  Each positive knot, 2-bridge knot, and +-adequate knot has a Legendrian representative 
  that
 does not have an exact, non-orientable  Langrangian endocobordism.  
 \end{cor}

   \begin{figure}
\begin{center}
	\includegraphics[height=2in]{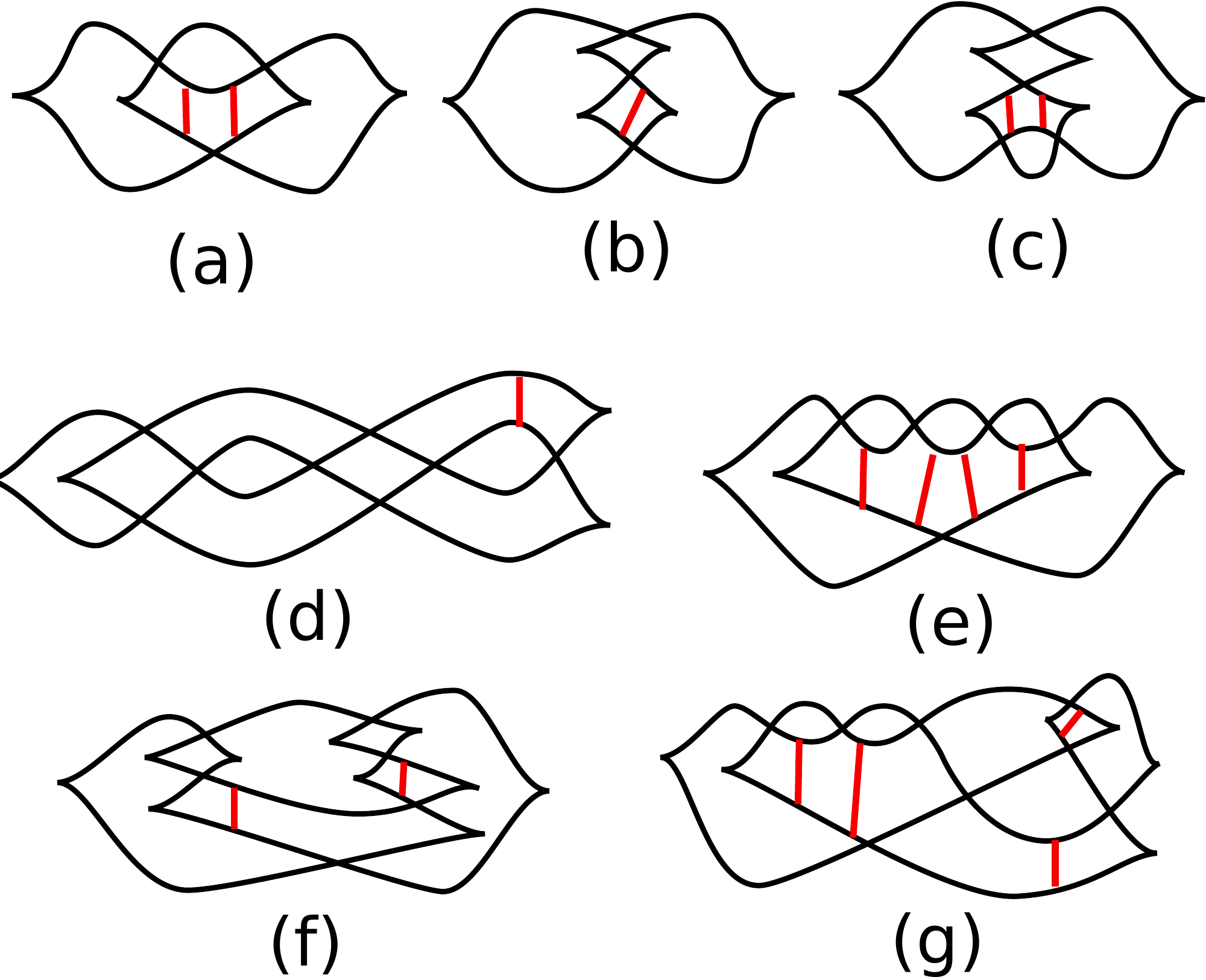}
\caption{Examples of Legendrians that do not have exact, non-orientable Lagrangian endocobordisms:  maximal $tb$ representatives of
(a) $m(3_1) = T(3,2) = K_{-2}$, (b) $3_1 = T(-3, 2) = K_1$,   
(c) $4_1 = K_2 = K_{-3}$,  (d) $5_1 = T(-5, 2)$, (e) $m(5_1) = T(5,2)$,  (f) $6_2$, and (g) $m(6_2)$.  The red lines indicate points
for surgeries.  
}
\label{fig:fillable}
\end{center}
\end{figure}

  Many maximal $tb$ representatives of  low crossing have fillings, orientable or not.  Figure
~\ref{fig:fillable} illustrates some Legendrians that can be verified to have exact, Lagrangian fillings: see Remark~\ref{rem:isotopy+surgery}.
 Many of the examples in Figure~\ref{fig:fillable} are Legendrian representatives of twist or torus knots.  Using
 Theorem~\ref{thm:fill-no-endo} together with
  classification results of Etnyre and Honda, \cite{etnyre-honda:knots}, and  Etnyre, Ng, and V\'ertesi, \cite{env},   
   we show that {\it all} maximal $tb$ represenatives of  twist knots,
  positive torus knots,  and negative torus knots of the form $T(-p, 2k)$, $p > 2k > 0$, do not have exact, 
  non-orientable Lagrangian endocobordisms:

  \begin{proof}[Proof of Corollary~\ref{cor:twist-torus}]  By Theorem~\ref{thm:fill-no-endo}, to show the non-existence
  of an exact, non-orientable Lagrangian endocobordism, it suffices to show the existence of an exact Lagrangian filling.
  
  First consider the case where $\leg$ is a maximal $tb$ representative of a twist knot, whose form is shown in 
  Figure~\ref{fig:top-twist}.              \begin{figure}
  \centerline{\includegraphics[height=.5in]{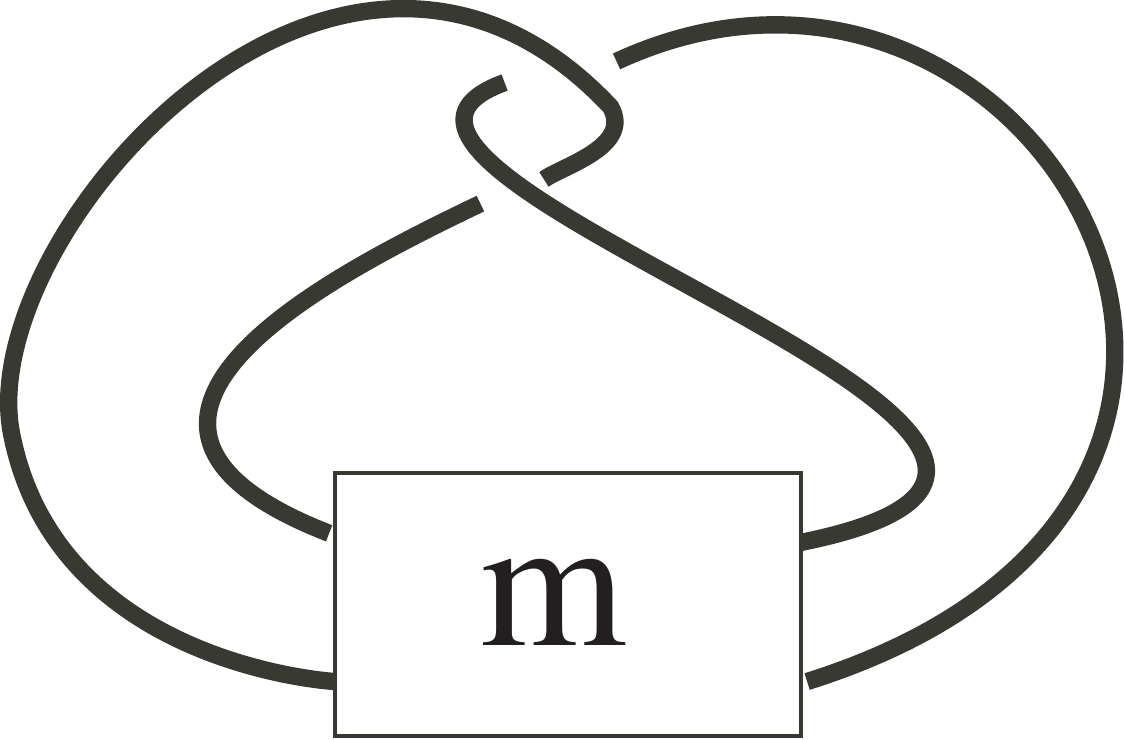}}
  \caption{The smooth twist knot $K_m$; the box contains $m$ right-handed half twists if
  $m \geq 0$, and $|m|$ left-handed twists if $m < 0$. Notice that $K_0$ and $K_{-1}$ are unknots.}
  \label{fig:top-twist}
\end{figure}
Etnyre, Ng, and V\'ertesi,  have  classified  all
  Legendrian twist knots, \cite{env}: every maximal $tb$ Legendrian  
  representative of $K_m$, for  $m \leq -2$, is Legendrian isotopic to one of the form
in Figure~\ref{fig:gen-neg-twist}, and  every maximal $tb$ Legendrian  
  representative  of  $K_m$, for $m \geq 1$,  
  is Legendrian isotopic to one of the form in Figure~\ref{fig:gen-pos-twist}. 
For a max $tb$ representative of a negative twist knot,
Figure~\ref{fig:gen-neg-twist} illustrates the two surgeries that show the existence of  an exact Lagrangian filling.
 For a max $tb$ Legendrian representative of a positive twist knot, the existence of an exact filling can be shown by an induction argument: 
Figure~\ref{fig:pos-twist-surg} (a), indicates surgery point when $m = 1$; for all $m \geq 1$, a maximal tb representative of $K_{m+1}$ can
be reduced to a  maximal $tb$ representative of $K_{m}$ by one surgery as indicated in Figure~\ref{fig:pos-twist-surg} (b).

   \begin{figure}
 \centerline{\includegraphics[height=.7in]{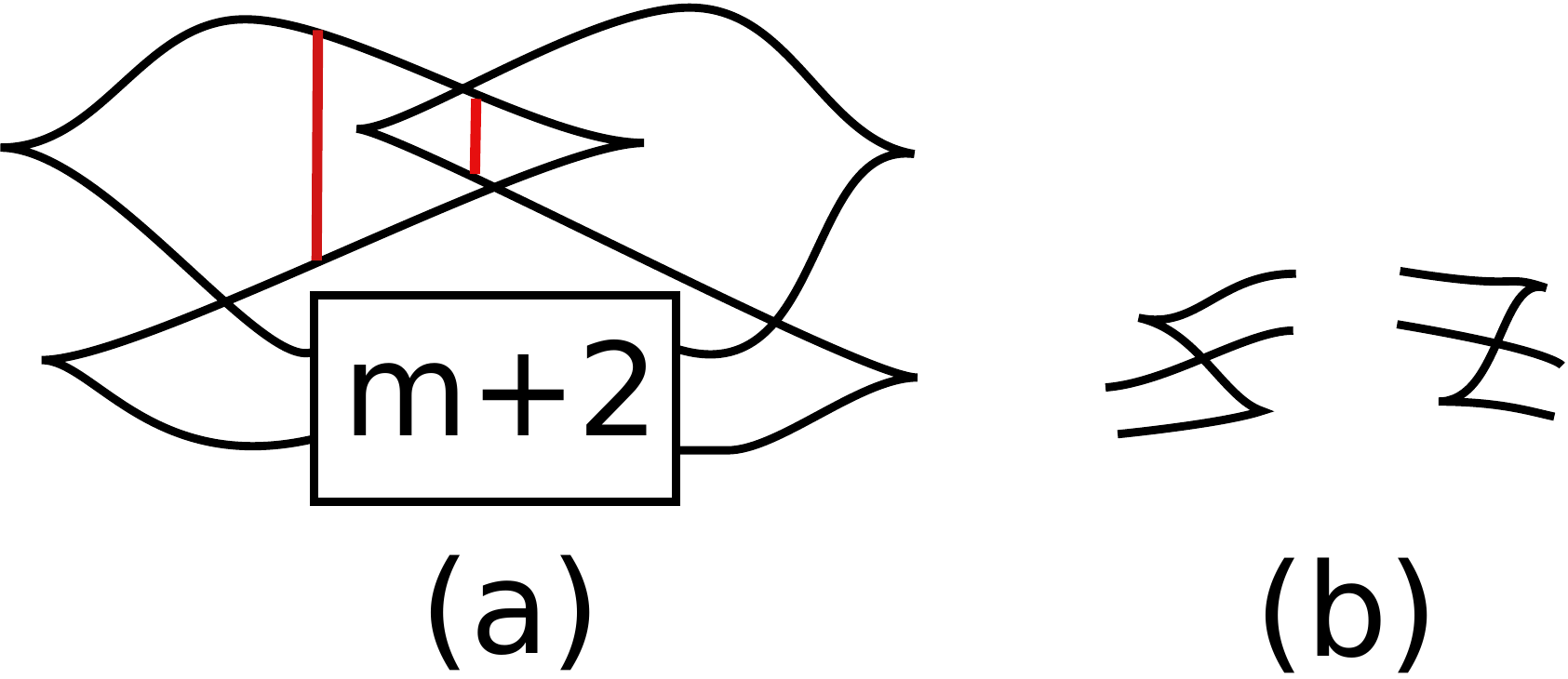}}
 \caption{Any maximal $tb$ Legendrian representative of a negative twist knot, $K_m$ with $m \leq -2$,  is Legendrian isotopic to one of the form in 
  (a) where the box contains $|m+2|$ half twists, each of form  $S$ as shown in   or of form $Z$ as
  shown in (b). Two surgeries produces a max $tb$ Legendrian unknot.}
  \label{fig:gen-neg-twist}
\end{figure}
 \begin{figure}
 \centerline{\includegraphics[height=.7in]{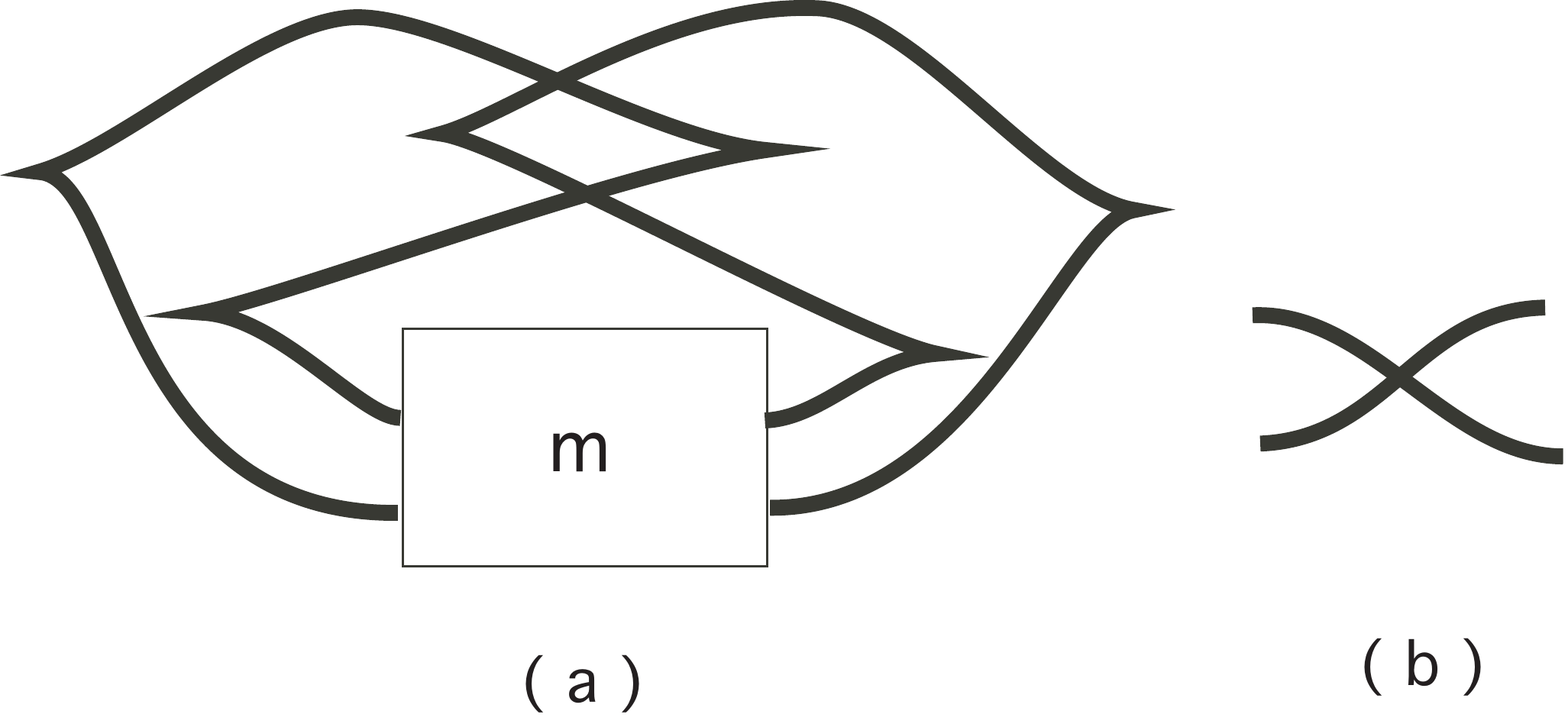}}
  \caption{Any maximal $tb$ Legendrian  representative of a positive twist knot, $K_m$ with $m \geq 1$,   is Legendrian isotopic to one of the form in 
  (a) where the box contains $m$ half twists, each of form $X$ as shown in (b).  }   \label{fig:gen-pos-twist}
\end{figure}
  \begin{figure}
 \centerline{\includegraphics[height=.7in]{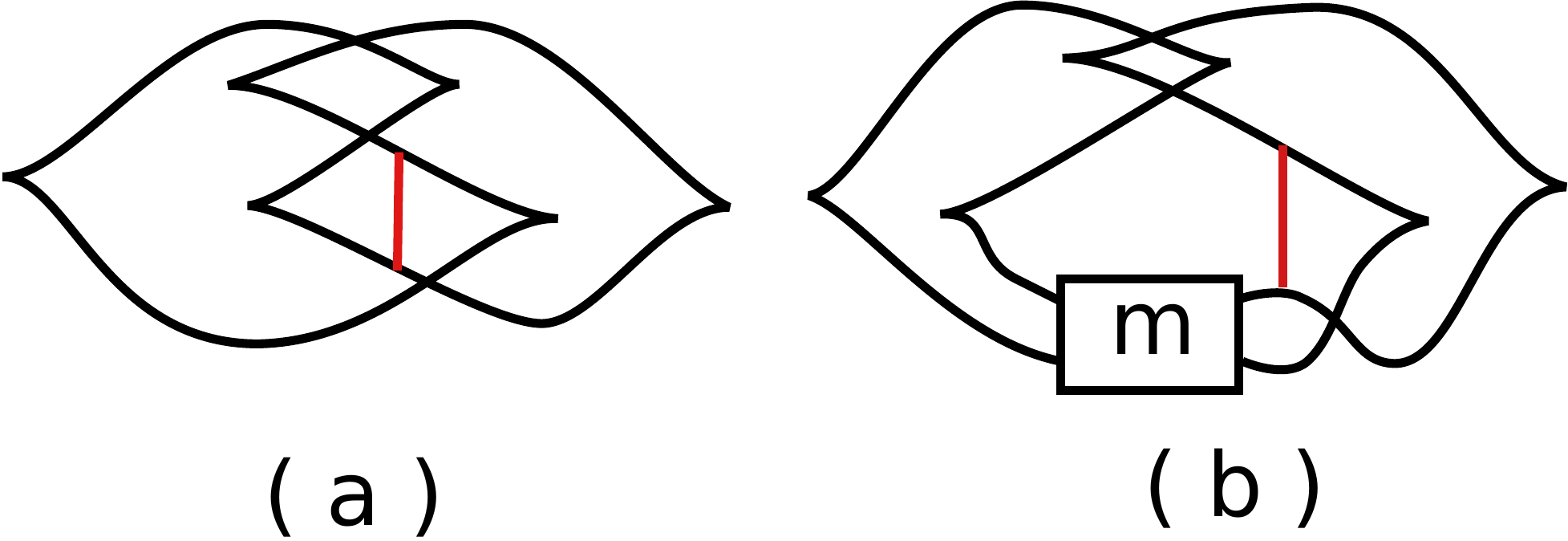}}
	  \caption{ An inductive argument shows that every max $tb$ representative of a positive twist knot has an exact Lagrangian filling.  
	  }
  \label{fig:pos-twist-surg}
\end{figure}

 Next consider maximal $tb$ Legendrian representatives of a torus knot,
 a knot that can be smoothly isotoped so that it lies on the surface of an unknotted torus in $\rr^3$.  Every torus
knot can be specified by a pair $(p,q)$ of coprime integers:  we will use the convention that 
the $(p,q)$-torus knot, $T(p,q)$, winds $p$ times around a meridonal curve of the torus and $q$ times in the longitudinal direction.
 In fact, $T(p, q)$ is equivalent to  $T(q, p)$ and to $T(-p, -q)$.  We will
 always assume that $|p| > q \geq 2$,  since we are interested in non-trivial torus knots.

 Etnyre and Honda, \cite{etnyre-honda:knots},  showed there is a unique  maximal $tb$ representative
 of a positive torus knot, $T(p,q)$ with $p > 0$.   The surgeries used  in \cite[Theorem 4.2]{bty}
show that each maximal representative is exactly fillable.  Figure~\ref{fig:pos-torus-surg} illustrates the orientable
surgeries for $(5,3)$-torus knot; in this sequence of surgeries,
   one begins with  surgeries on the innermost strands, and then
   performs a Legendrian isotopy so that it is possible to do a surgery on the next set of innermost strands.
 \begin{figure}
 \centerline{\includegraphics[height=.5in]{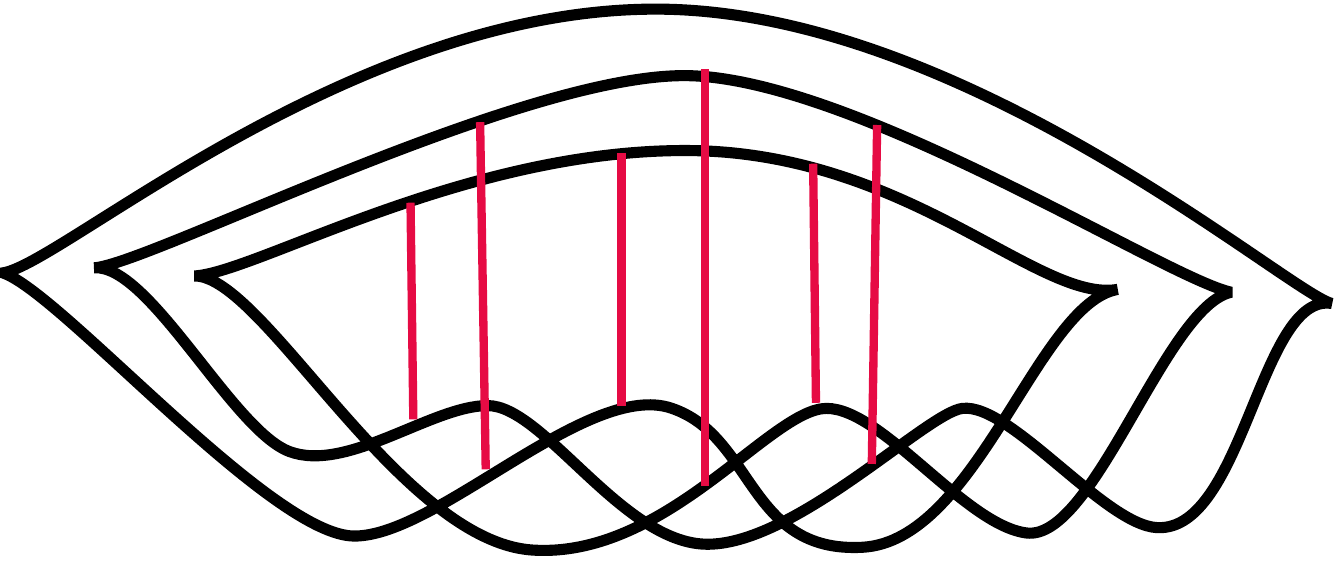}}
	  \caption{Surgeries that result in an exact filling of the maximal $tb$ representative of the positive torus knot $T(5,3)$.
	 }
  \label{fig:pos-torus-surg}
\end{figure}
  
 Lastly consider the case where $\Lambda$ is topologically a negative torus knot, $T(-p, 2k)$ with $p > 2k > 0$. In this case,
 Etnyre and Honda have shown that 
the number of different maximal $tb$ Legendrian representations  
 depends on the divisibility of $p$ by $2k$:  if
    $|p| = m (2k) + e$, $0 < e < 2k$,       
    there are $m$ non-oriented
     Legendrian representatives of $T(-p, 2k)$ 
     with maximal $tb$.
         These different representatives  with maximal $tb$
    are obtained by writing $m =  1 + n_1 + n_2$,  where $ n_1, n_2 \geq 0$,
    and  then $\Lambda_{(n_1, n_2)}$ is constructed using the form shown in Figure~\ref{fig:neg-torus-unknot} with $n_1$ and $n_2$ copies of the tangle  $B$ inserted as indicated;
    this figure also shows $k$ surgeries that guarantee the existence of an exact Lagrangian filling.  \end{proof}

      \begin{figure}
 \centerline{\includegraphics[height=1.5in]{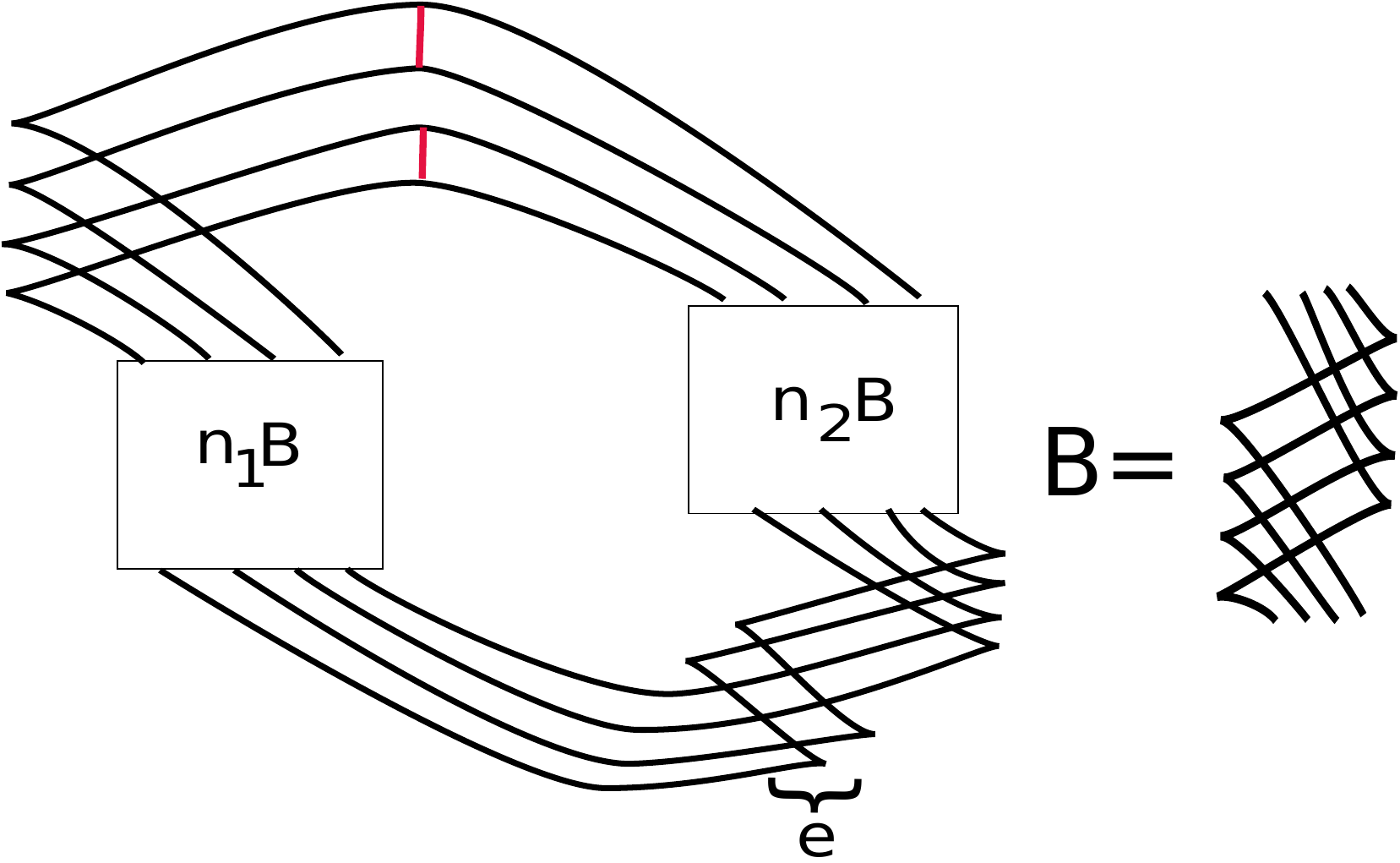}}
  \caption{The general form of a maximal $tb$ representative of a  negative torus knot $T(-p, 2k)$, with $p > 2k > 0$, with $k = 2$
  and $|p| = (1 + n_1 + n_2)(2k) + e$; $k$ surgeries produce a trivial Legendrian link of
  maximal $tb$ unknots.}
  \label{fig:neg-torus-unknot}
\end{figure}

Some  comments on obstructions to exact fillings are discussed in Section~\ref{sec:questions}.

\section{Constructions of Exact, Non-orientable Lagrangian Cobordisms}

In this section, we will construct an exact, non-orientable Lagrangian endocobordisms of crosscap genus $4$ for any stabilized Legendrian knot,
and a non-orientable Lagrangian cobordism between any two stabilized Legendrian knots.  All these exact Lagrangian cobordisms are
constructed through isotopy and surgery, see Remark~\ref{rem:isotopy+surgery}.

 Central to  these constructions will be the following lemma, which says that with respect to either orientation on  $\leg_+$ one can always introduce a pair of oppositely oriented
zig-zags, and if one has a pair of oppositely oriented zig-zags in $\leg_+$, then one can remove either element of this pair; see
Figure~\ref{fig:stab-destab}.

\begin{figure} 
\begin{center}
	\includegraphics[height=1in]{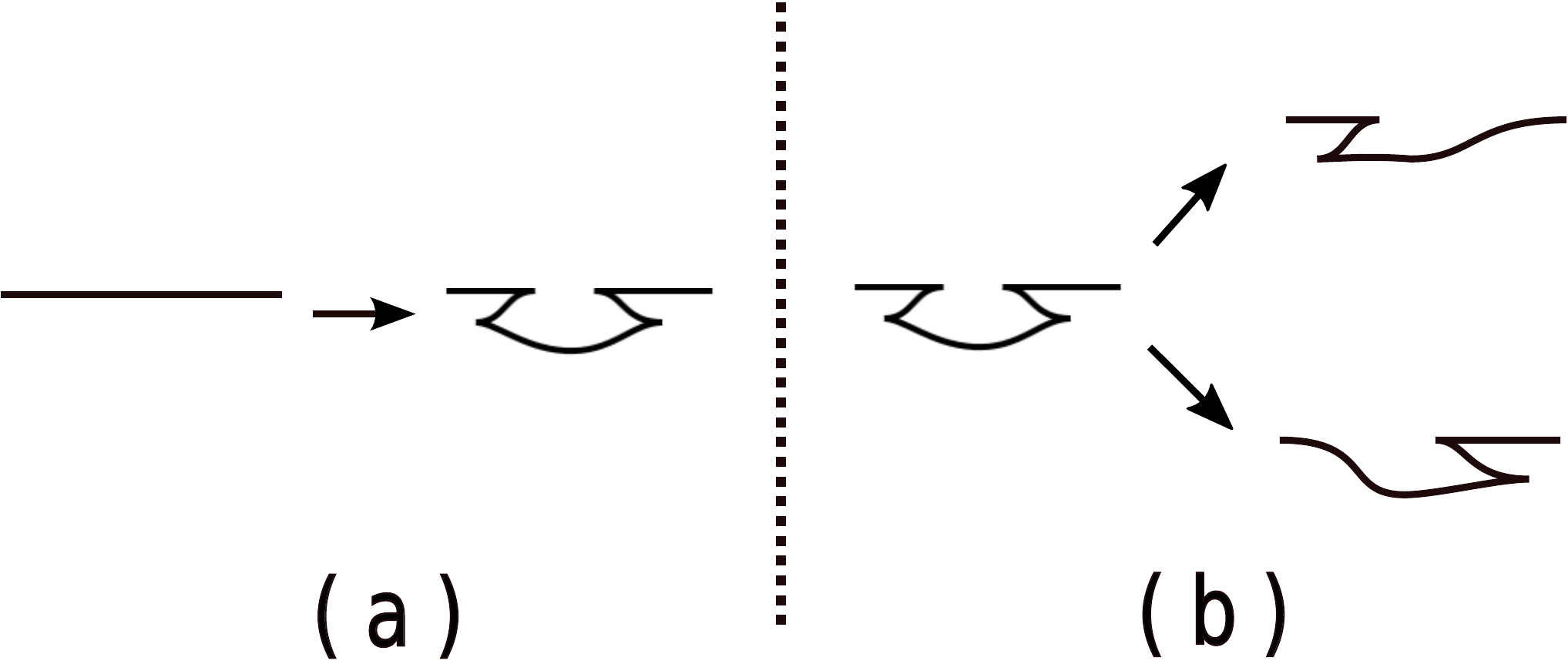}
\caption{Via isotopy and surgeries, at least one of which is non-orientable, it is possible to construct exact non-orientable Lagrangian
cobordisms between (a) $\leg_+ = \leg$ and $S_-S_+(\leg)$,  (b) $\leg_+ = S_-S_+(\leg)$  and
$\leg_- = S_+(\leg)$ or $\leg_- = S_-(\leg)$. } 
\label{fig:stab-destab}
\end{center}
\end{figure}

 \begin{lem}\label{lem:stab-destab}  Let $\leg$ be any oriented Legendrian knot. Then there exists an exact, non-orientable Lagrangian cobordism:
\begin{enumerate}
\item  of crosscap genus $2$ between $\leg_+ = \leg$ and
$\leg_- = S_-S_+(\leg)$;
\item of crosscap genus $1$ between $\leg_+ = S_-S_+(\leg)$  and
$\leg_- = S_+(\leg)$ or $\leg_- = S_-(\leg)$.
\end{enumerate}
\end{lem}

\begin{rem} With non-orientable cobordisms,
given an orientation on $\leg_+$, there is no canonical orientations for $\leg_-$.  In Lemma~\ref{lem:stab-destab},
an orientation on $\leg_+$ is chosen so that there are well-defined $S_-(\leg)$ and $S_+(\leg)$, but the statement
implies that  $\leg_-$ can be $S_-S_+(\leg)$, $S_+(\leg)$, or $S_-(\leg)$ with either orientation.

\end{rem}

\begin{proof} The strategy will be to construct the desired exact, non-orientable Lagrangian cobordism via
Legendrian isotopy and surgeries that are performed on a portion of a strand.  Figure~\ref{fig:adding-doub-stab} 
illustrates the isotopy and surgeries, the second  of which is non-orientable, that implies the existence of a
crosscap genus $2$ Lagrangian cobordism between $\leg_+ = \leg$ and $\leg_- = S_-S_+(\leg)$.  Figure~\ref{fig:stab-removal}
illustrates the isotopy and surgery that implies the existence of 
 a crosscap genus $1$ Lagrangian cobordism between $\leg_+ = S_-S_+(\leg)$ and 
 $\leg_- = S_+(\leg)$, when the original strand is oriented from right to left, 
 or to   $\leg_- = S_-(\leg)$, when the original strand is oriented from left to right. \end{proof}

\begin{figure} 
\begin{center}
\includegraphics[width=3.in]{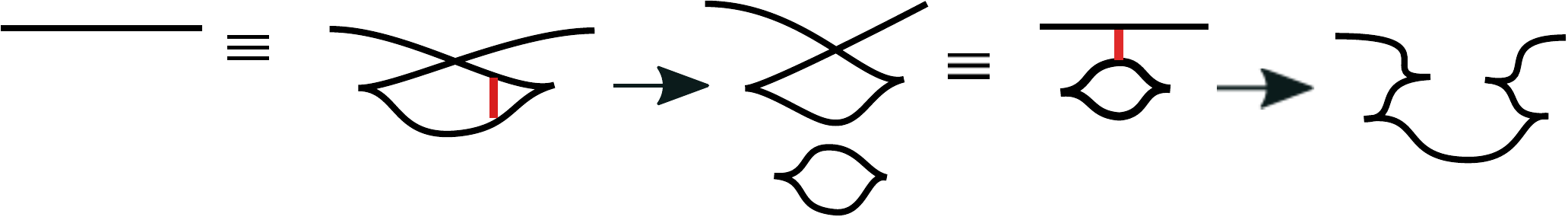}
\caption{By applying an orientable and a non-orientable surgery, any strand can have a pair of oppositely oriented zig-zags introduced.}
\label{fig:adding-doub-stab}
\end{center}
\end{figure}

\begin{figure} 
\begin{center}
\includegraphics[width=3in]{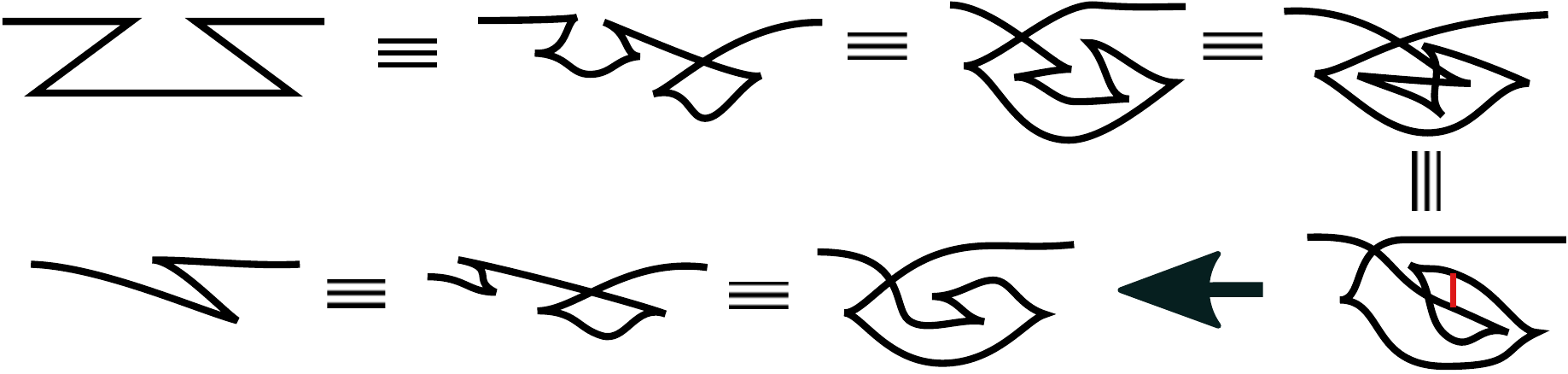}
\caption{    In the presence of oppositely oriented zig-zags, via one non-orientable surgery, one of the zig-zags can be removed.  }
\label{fig:stab-removal}
\end{center}
\end{figure}

\subsection{Exact, Non-Orientable Lagrangian Endocobordisms}

In Theorem~\ref{thm:fill-no-endo}, it was shown that Legendrians that are exactly fillable do not have exact, non-orientable Lagrangian
endocobordisms.  However exact, non-orientable Lagrangian endocobordisms do exist for stabilized knots:

\begin{proof}[Proof of Theorem~\ref{thm:reflexive}] First consider the case where $\leg$ is the negative stabilization of a Legendrian:  $\leg = S_-(\widehat \leg)$.
Then by applying Lemma~\ref{lem:stab-destab}, there exists an exact, non-orientable Lagrangian cobordism:
\begin{enumerate}
\item of  crosscap genus $2$ between $\leg$ and $S_-S_+(\leg)$;
\item of crosscap genus $1$ between $S_-S_+(\leg)$ and $S_+(\leg)$;
\item of crosscap genus $1$ between  $S_+(\leg) = S_+(S_-(\widehat \leg))$ and $S_-(\widehat \leg) = \leg$.
\end{enumerate}
Stacking these cobordisms results in an exact, non-orientable Lagrangian endocobordism of crosscap genus $4$.  Additional stacking
results in arbitrary multiples of crosscap genus $4$.  

An analogous argument proves the case where $\leg$ is the positive stabilization of a Legendrian:  $\leg = S_+(\widehat \leg)$.
\end{proof}

\subsection{Exact, Non-Orientable Lagrangian Cobordisms between Stabilized Legendrians}

Given that every stabilized Legedendrian knot has a non-orientable Lagrangian endocobordism,
a natural question is: What Legendrian knots can appear as a ``slice" of such an endocobordism?
In this section, we show that {\it any} stabilized Legendrian knot     can appear as such a slice.  
 
\begin{thm} \label{thm:different-knots-path}   For smooth knot types $K, K'$, let $\leg$ be any Legendrian
representative of $K$ and let $\leg'$ be a stabilized Legendrian
representative of $K'$.
  Then  there exists an exact, non-orientable Lagrangian cobordism between $\leg_+ = \leg$ and $\leg_- = \leg'$.
\end{thm}

  Before moving to the proof of Theorem~\ref{thm:different-knots-path}, we show that non-orientable Lagrangian cobordisms define an equivalence
relation on the set of stabilized Legendrian knots:

\begin{proof}[Proof of Theorem~\ref{thm:equiv}]   Let $\mathcal L^s$ denote the set of all stabilized Legendrian knots of any
smooth knot type.  Define
the relation $\sim$ on $\mathcal L^s$ by $\leg_1 \sim \leg_2$ if there exists an exact, non-orientable Lagrangian cobordism
from $\leg_+ = \leg_1$ to $\leg_- = \leg_2$.  Reflexivity of $\sim$ follows from Theorem~\ref{thm:reflexive}.  Symmetry of
$\sim$ follows from Theorem~\ref{thm:different-knots-path}.  Transitivity of $\sim$ follows from Lemma~\ref{lem:gluing}. Thus $\sim$ defines an equivalence relation.  Moreover, by Theorem~\ref{thm:different-knots-path},
we see that with respect to this equivalence relation, there is only one equivalence class.
\end{proof}  

To prove Theorem~\ref{thm:different-knots-path}, it will be useful to first show that there is an
exact, non-oriented Lagrangian cobordism between any  two stabilized Legendrians of a fixed knot type:

\begin{prop} \label{prop:fixed-knot-path} Let $K$ be any smooth knot type, and let $\leg, \leg'$ be Legendrian representatives of $K$
where $\leg'$ is   stabilized.
Then there exists an exact, non-orientable Lagrangian cobordism between $\leg_+ = \leg$ and $\leg_- =  \leg'$.
\end{prop}

\begin{proof} Fix a smooth knot type $K$, and  let $\leg_1, \leg_2$ be Legendrian representatives where $\leg_2$ is stabilized.  
By results of Fuchs and Tabachnikov, \cite{f-t}, we know that there exists $r_1, \ell_1, r_2, \ell_2$ so that $S_-^{\ell_1} S_+^{r_1}(\leg_1) = S_-^{\ell_2} S_+^{r_2}(\leg_2)$.  
By applying additional positive stabilizations, if needed, we can assume $r_1 > \ell_1$.

Consider the case where $\leg_2$ is the negative stabilization of some Legendrian: $\leg_2 = S_-(\hat\leg_2)$.  
By applications of Lemma~\ref{lem:stab-destab},  there exists an exact, non-orientable Lagrangian cobordism between:
\begin{enumerate}
\item   $\leg_1$ and $S_-^{r_1} S_+^{r_1}(\leg_1)$;
\item   $S_-^{r_1} S_+^{r_1}(\leg_1)$ and  $S_-^{\ell_1} S_+^{r_1}(\leg_1)$, and thus between
$S_-^{r_1} S_+^{r_1}(\leg_1)$ and $S_-^{\ell_2} S_+^{r_2}(\leg_2)$;
\item  $S_-^{\ell_2} S_+^{r_2}(\leg_2)$ and  $S_+^{r_2}(\leg_2)$;
\item   $S_+^{r_2}(\leg_2) = S_+^{r_2}(S_-(\hat\leg_2)) $ and $S_-(\hat\leg_2) = \leg_2$.
\end{enumerate}
By stacking these cobordisms (Lemma~\ref{lem:gluing}), we have our desired exact, non-orientable Lagrangian cobordism between $\leg_1$ and $\leg_2$.  
An analogous argument proves the case where $\leg_2$ is the positive stabilization of some Legendrian.
\end{proof}

\begin{proof}[Proof of Theorem~\ref{thm:different-knots-path}]  The strategy here is to first show that one can construct an 
exact, non-orientable Lagrangian cobordism between $\leg$
and a stabilized Legendrian unknot $\leg_0$.  Similarly, it is possible to construct an exact, non-orientable Lagrangian cobordism between $\leg'$
and a stabilized Legendrian unknot $\leg_0'$; we will show it is possible to ``reverse" this sequence of surgeries and construct an
exact, non-orientable Lagrangian cobordism between $\leg_0'$
and $\widetilde \leg'$, which is a stabilization of $\leg'$.   By Proposition~\ref{prop:fixed-knot-path}, there exists 
an exact, non-orientable Lagrangian cobordism between $\leg_0$ and $\leg_0'$ and   between 
  $\widetilde \leg'$ and $\leg'$.  Thus by stacking, we will have the desired exact, non-orientable Lagrangian 
cobordism between $\leg$ and $\leg'$.

We first show how it is possible to construct an exact, non-orientable Lagrangian cobordism from $\leg$ to a Legendrian unknot; cf.,  \cite{bty}.  
Let $\leg$ be an arbitrary stabilized Legendrian knot. We can assume that $\leg$ has
at least one positive crossing by, if necessary, applying a Legendrian Reidemeister 1 move.  As shown in Figure~\ref{fig:cross-removal},
performing an orientable or non-orientable surgery near a crossing  produces a crossing  that can be removed through Legendrian Reidmeister moves.   Perform such a surgery on every crossing in $\leg$ until you have obtained $k$ disjoint stabilized Legendrian unknots;  since $\leg$ has at least one positive crossing, we have performed at least one 
non-orientable surgery.   Align the $k$ Legendrian unknots vertically and perform orientable or non-orientable surgeries so that we obtain a single stabilized Legendrian unknot $\leg_0$.  In this way, we have constructed an exact, non-orientable Lagrangian cobordism between
$\leg$ and $\leg_0$.

 \begin{figure}[h]
 \centerline{\includegraphics[width=3in]{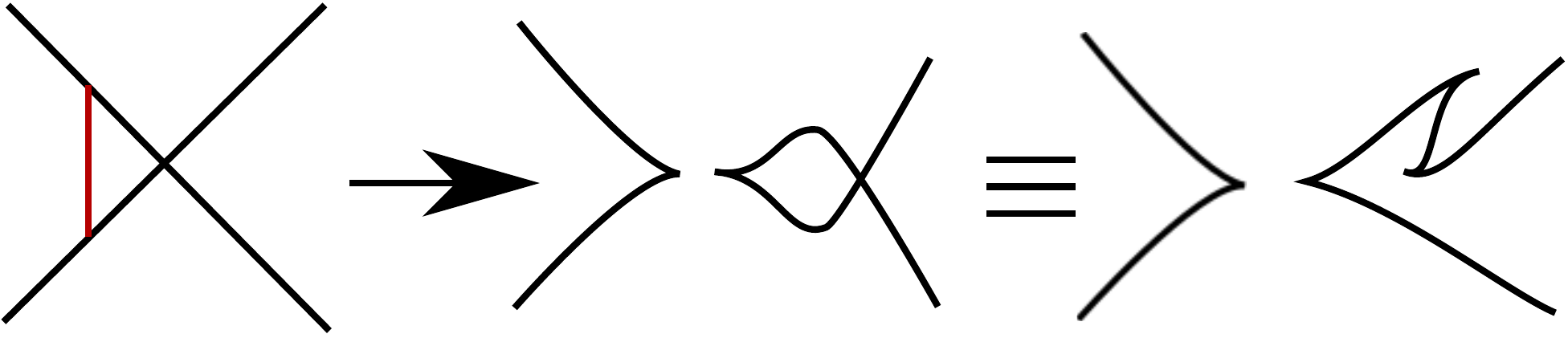}}
  \caption{For any Legendrian knot $\leg$, perform a surgery near each crossing in order to get a disjoint set of Legendrian unknots.
   }\label{fig:cross-removal}
\end{figure}
 \begin{figure}[h]
 \centerline{\includegraphics[width=3in]{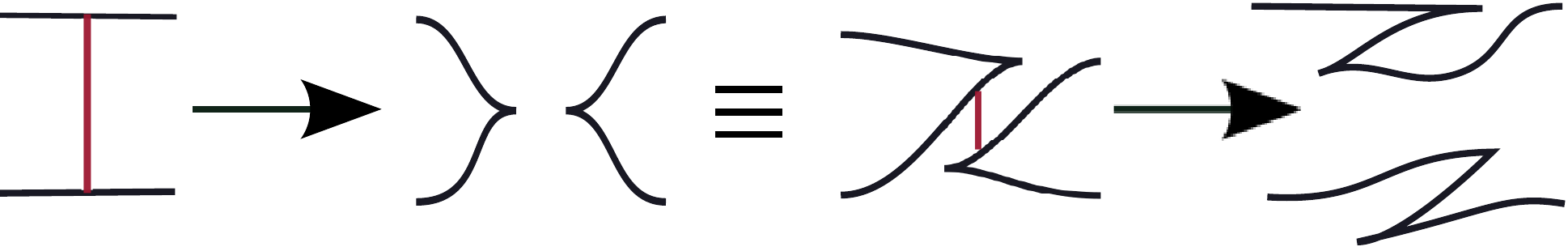}}
  \caption{Surgeries used to convert to a link of Legendrian unknots can be ``undone", at the cost of additional stabilizations.}\label{fig:surg-reverse}
\end{figure}

A similar procedure can be used to construct a sequence of surgeries from $\leg'$ to another Legnedrian unknot $\leg_0'$; 
now we show it is possible to ``reverse" this procedure  and construct a sequence of surgeries from   $\leg_0'$ to $\widetilde \leg'$, a Legendrian obtained by applying stabilizations to $\leg'$.   Figure~\ref{fig:surg-reverse} illustrates how every surgery that was used to get to a Legendrian unknot
can be undone at the cost of adding additional zig-zags into the original strands.  
Figure~\ref{fig:unknotting-knotting} illustrates this procedure in a particular example. 
 
As outlined at the beginning at the proof, these constructions  
prove the existence of an exact Lagrangian cobordism from $\leg_+ = \leg$ to $\leg_- = \leg'$. 
\end{proof}

\begin{figure}[h]
 \centerline{\includegraphics[width=3.5in]{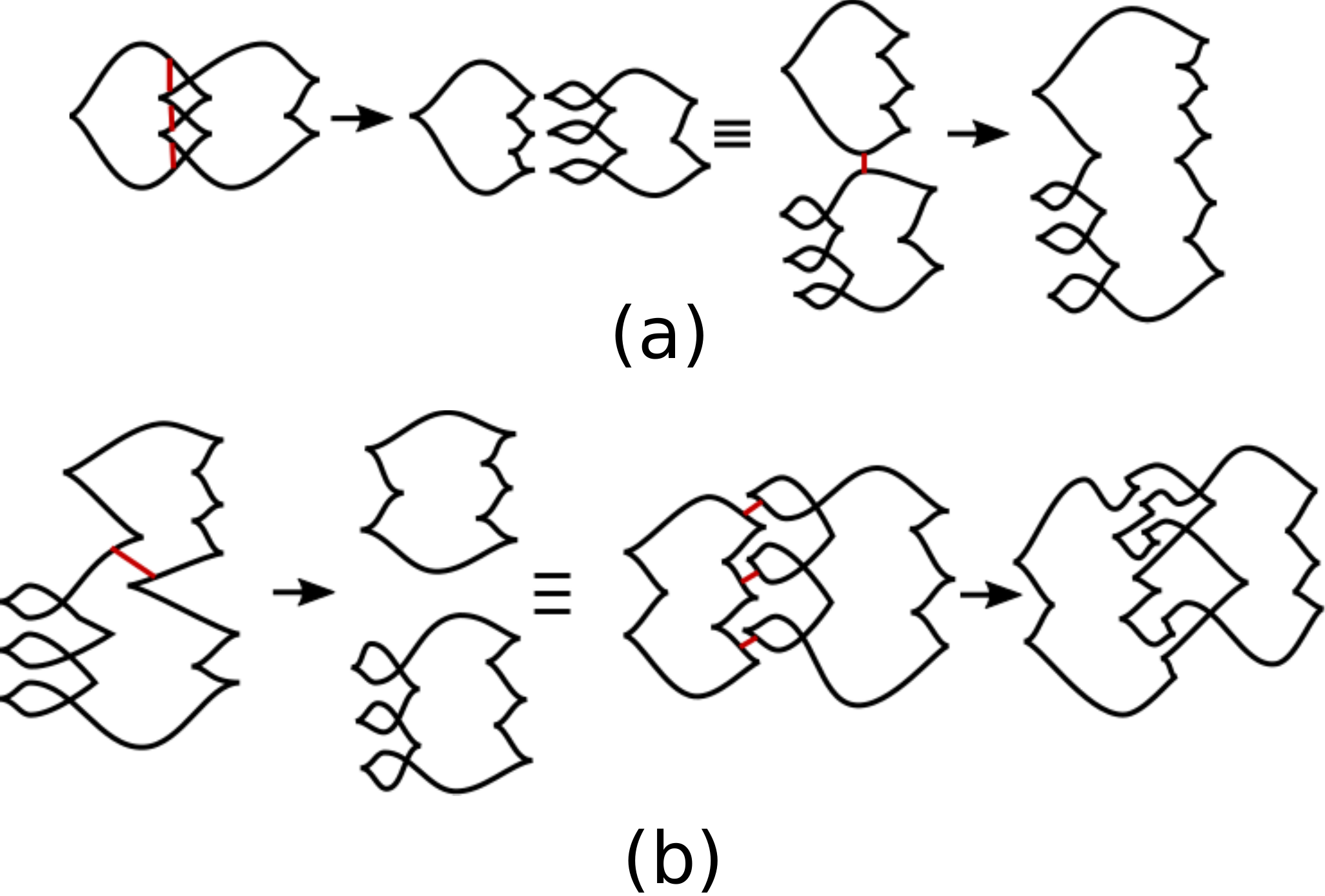}}
  \caption{(a) Surgeries that give rise to an exact non-orientable Lagrangian cobordism from the max $tb$ version of $3_1$ to a stabilized unknot.
  (b) Surgeries that give rise to an exact non-orientable Lagrangian cobordism from the stabilized unknot to  a stabilized representative of $3_1$.}
  \label{fig:unknotting-knotting}
\end{figure}

\section{Additional Questions}\label{sec:questions}
We end with a brief discussion of some additional questions.

From results above, we know that exactly fillable Legendrian knots do not admit exact, non-orientable Lagrangian endocobordisms
while stabilized Legendrian knots do.  There are examples of Legendrian knots that are neither exactly fillable nor stabilized.
As mentioned above, Ekholm, \cite{ekholm:rsft}, has shown that if $\leg$ is exactly fillable, then there exists an ungraded augmention of $\mathcal A(\leg)$.
By work of Sabloff, \cite{rulings}, and independently, Fuchs and Ishkhanov, \cite{fuchs-ishk}, we then know that there exists
an ungraded ruling of $\leg$.  Then it follows by work of Rutherford, \cite{rutherford:kauffman}, that the Kauffman bound on the
 maximal $tb$ value for all Legendrian representatives of
  the smooth knot type of $\leg$ is sharp.  Thus, if the Kauffman bound is not sharp for the smooth knot type
 $K$, then no Legendrian representative of $K$ is exactly fillable.  

 \begin{ques} \label{ques:maxtb-no-fill} If $\leg$ is a maximal $tb$ representative of a knot type $K$ for which
 the upper bound on $tb$ for all Legendrian representatives given by the Kauffman polynomial is not sharp, does $\leg$ have an exact,
 non-orientable Lagrangian endocobordism?  
 \end{ques}
\noindent
The Legendrian representative of $m(8_{19})$ mentioned in Question~\ref{ques:m(819)} satisfies the hypothesis in Question~\ref{ques:maxtb-no-fill}.
A list of some additional smooth knot types where the Kauffman bound is not sharp can be found in \cite[Section 4]{lenny:khovanov}.

There are also examples of Legendrians with non-maximal $tb$ that are not stabilized.  For example, $m(10_{161})$ is a knot
type where the unique maximal $tb$ representative has a  filling.  However, there are Legendrian representatives with non-maximal
$tb$ that do not arise as a stabilization.  As shown in  \cite[Figure 1]{clay-shea}, this Legendrian does have an ungraded ruling.
\begin{ques}  Does the non-stabilized, non-maximal $tb$ Legendrian representative of $m(10_{161})$ have an
exact, non-orientable Lagrangian endocobordism?
\end{ques}
\noindent
Additional examples of non-stabilized and non-maximal $tb$ representatives can be found in the Legendrian
knot atlas of Chongchitmate and Ng, \cite{ng:atlas}.

There are additional questions that arise from the constructions of fillings. For example,    
 it is known  by results of Chantraine, \cite{chantraine}, that orientable fillings realize the smooth $4$-ball genus.   In Figure~\ref{fig:fillable},
examples are given of non-orientable Lagrangian fillings of maximal $tb$ representatives of $6_2$ and $m(6_2)$ of
crosscap genus $2$ and $4$: the smooth 4-dimensional crosscap number of both $6_2$ and $m(6_2)$ is $1$.
 
\begin{ques}     Does there exist
a non-orientable Lagrangian filling of these Legendrian representatives of $6_2$ and $m(6_2)$ of crosscap genus $1$?
\end{ques}

\bibliographystyle{amsplain} 
\bibliography{main-no}

\end{document}